\documentclass[11pt]{amsart}
\usepackage{amssymb, latexsym, amsmath, amsfonts, hyperref, enumitem, fancyhdr,amsthm}
\usepackage{mathrsfs}
\usepackage[titletoc,toc,title]{appendix}
\usepackage[normalem]{ulem}
\usepackage[T1]{fontenc}

\theoremstyle{plain}
\newtheorem{thm}{Theorem}[section]
\newtheorem*{theorem*}{Theorem}

\newtheorem{lem}[thm]{Lemma}
\newtheorem{prop}[thm]{Proposition}

\theoremstyle{definition}

\newtheorem{defn}[thm]{Definition}
\newtheorem{rem}[thm]{Remark}

\newtheorem{Qn}[thm]{Question}

\newtheorem*{cfpbl}{\tt CF problem}

\hypersetup{
  colorlinks,
  citecolor=black,
  linkcolor=black,
  urlcolor=blue}
\theoremstyle{plain}

\newcommand{\D}{\mathbb{D}}
\newcommand{\C}{\mathbb{C}}

\newcommand{\h}{\mathbb{H}}

\newcommand{\N}{\mathbb{N}}
\newcommand{\T}{\mathbb{T}}
\newcommand{\Z}{\mathbb{Z}}

\newcommand{\CF}{Carath\'{e}odory--Fej\'{e}r interpolation }
\newcommand{\VZ}{\boldsymbol{z}}

\hypersetup{
  colorlinks,
  citecolor=blue,
  linkcolor=blue,
  urlcolor=blue}

\title[Carath\'{e}odory--Fej\'{e}r interpolation problem]{The Carath\'{e}odory--Fej\'{e}r interpolation problem\\ for the polydisc}

\author{Rajeev Gupta}
\author{Gadadhar Misra}

\address[G. Misra]{Department of Mathematics, Indian Institute of Science, Bangalore - 560012, India}
\address[R. Gupta]{Department of Mathematics and Statistics\\
Indian Institute of Technology\\
Kanpur - 208016}

\email[R. Gupta]{rajeevg@iitk.ac.in}

\email[G. Misra]{gm@math.iisc.ernet.in}

\begin{document}

\thanks{The first author was supported, in part,  by  the National Board for Higher Mathematics. The second author was supported, in part, through the J C Bose National Fellowship and UGC-SAP IV}

\thanks{The results of this paper are from the PhD thesis of the first author submitted to the Indian Institute of Science in the year 2015.}

\pagestyle{headings}

\begin{abstract}
We give an algorithm for finding a solution 
to the Carath\'{e}odory--Fej\'{e}r interpolation problem 
on the polydisc $\D^n,$ whenever it exists. 
A necessary condition for the existence 
of a solution becomes apparent from this algorithm. 
A generalization of the well-known theorem due to Nehari
has been obtained. A proof of the Kor\'{a}nyi--Puk\'{a}nszky 
theorem is given using the spectral theorem.  

\end{abstract}

\maketitle

\section{Introduction}
For a holomorphic map $h$ on the polydisc $\mathbb D^n:= \{\VZ:=(z_1,\ldots , z_n)\in \mathbb C^n: |z_i| < 1,\,1\leq i\leq n\},$ and any multi-index $I,$ set $h^{(I)}(z)= \big (\partial_1^{i_1} \cdots \partial_n^{i_n} h\big ) (\VZ),$ $\VZ \in \mathbb D^n.$ We recall below a version of the well-known Carath\'{e}odory--Fej\'{e}r interpolation problem. 
\addtocounter{subsection}{1}
 \subsection*{Carath\'{e}odory--Fej\'{e}r interpolation problem {\tt (CF problem (n, d) )}} 
Given any polynomial $p$ in $n$ - variables 
of degree $d,$ 
find necessary and sufficient conditions 
on the coefficients of $p$ to 
ensure the existence of a holomorphic function $h$ 
defined on the polydisc $\mathbb D^n$ such that 
$f:=p+h$ maps the polydisc $\mathbb D^n$ 
	into $\mathbb D$ and that for any multi-index $I$  with  $|I| \leq d,$ $h^{(I)}(\boldsymbol 0)=0.$ 

An explicit solution to the CF problem has been found in the case of $n=1,$ see \cite[p. 179]{Nik}.  More recently, several results (cf. \cite{BLTT,FFE,EPP,Woe,HMLZ}) have been obtained for the solution to the CF problem for $n>1.$  In this article, we present a reformulation of the {\tt CF problem (n,d)}. It involves  finding $d+1$ polynomials $p_0, \ldots, p_d$ from the polynomial $p$ given in the {\tt CF problem (n,d)} according to a well-defined and explicit rule. The reformulation asks for the existence of a contractive holomorphic function  $f: \mathbb D \to \mathcal{B}(L^2(\mathbb{T}^{n-1})),$ where $\mathbb T$ is the unit circle and $\mathbb T^{n-1}$ is the $(n-1)$ - fold cartesian product of $\mathbb T,$ extending the polynomial 
$P(z)=M_{p_0}+M_{p_1}z+\cdots + M_{p_d}z^d.$ The precise statement follows. 

%

\subsection*{Reformulation of the Carath\'{e}odory--Fej\'{e}r interpolation problem {\tt (CF problem (R))}} 
Let $P:\mathbb D\to \mathcal{B}(L^2(\mathbb{T}^{n-1}))$ be a polynomial of the form $$P(z)=M_{p_0}+M_{p_1}z+\cdots + M_{p_d}z^d,\,\,p_k\in\mathscr{M}^{(k)}_{n-1},$$ where 
$$\mathscr{M}^{(k)}_{n-1}:= \mbox{\rm span}\big \{z_1^{\alpha_1}z_2^{\alpha_2}\cdots z_{n-1}^{\alpha_{n-1}}:0\leq \alpha_1\leq\cdots \leq \alpha_{n-1},\, \alpha_1 + \cdots +\alpha_{n-1} \leq (n-1)k\big \},$$ and $M_{p_k}$ is the multiplication operator on $L^2(\mathbb{T}^{n-1}).$  Find necessary and sufficient condition on $p_0,\ldots,p_d$ ensuring the  existence of $p_{\ell}\in\mathscr{M}^{(\ell)}_{n-1}$ for each $\ell > d$ such that $f(z):=\sum_{s=1}^{\infty}M_{p_s}z^s$ maps $\mathbb{D}$ into the unit ball of $\mathcal{B}(L^2(\mathbb{T}^{n-1})).$

We show that the polynomials $p_0, \ldots, p_d,$ $p_k \in \mathscr{M}^{(k)}_{n-1},$ $0\leq k \leq d,$ determine uniquely a polynomial $p$ in $n$ - variables of degree $d$ and vice-versa making {\tt CF problem (R)} a reformulation of the {\tt CF problem (n,d)}.

Our method, in general, gives a (explicit) necessary condition 
for the existence of a solution to 
the {\tt CF problem (n,d)}. For instance, let $p$ be a polynomial in two variables of degree $2,$ $p(\boldsymbol 0)=0.$ Set 
\begin{equation}\label{1p1p2}
p_1(z)=\frac{\partial p}{\partial z_1}(\boldsymbol 0)+\frac{\partial p}{\partial z_2}(\boldsymbol 0)z
\mbox{ and }
p_2(z)=\frac{1}{2}\frac{\partial^2 p}{\partial z_1^2}(\boldsymbol 0)+\frac{\partial^2 p}{\partial z_1\partial z_2}(\boldsymbol 0)z+\frac{1}{2}\frac{\partial^2 p}{\partial z_2^2}(\boldsymbol 0)z^2.
\end{equation}
In this case, we show that $|p_1(z)|^2+|p_2(z)|\leq 1,$ $z\in \mathbb D,$ 
(this is abbreviated to $|p_1|^2+|p_2|\leq 1$) 
is a necessary condition for the existence of a solution to 
the {\tt CF problem (2,2)}.
By means of an example, we show 
that this necessary condition is not sufficient.  
For the {\tt CF problem (2,2)}, we  isolate a class of polynomials 
for which our necessary condition is also sufficient. 
This is verified using a deep theorem of 
Nehari reproduced below (cf. \cite[Theorem 15.14]{NY}). 

We also give an algorithm, in Section \ref{CFP}, for constructing 
a solution whenever such a solution exists. 
The algorithm involves finding, 
inductively, polynomials $p_k$ in $\mathscr{M}^{(k)}_{n-1}$ 
such that a certain block Toeplitz operator, 
made up of multiplication operators corresponding 
to these polynomials is contractive. 
A solution to the {\tt CF problem (n,d)} 
exists if and only if this process is completed successfully. 
If $n=1$ and the necessary condition we have obtained is met, 
then the algorithm completes successfully 
and produces a solution to 
the CF problem. 
Thus in this case, 
we fully recover the solution to 
the {\tt CF problem (1,d)}. 
In Section \ref{Nehari}, 
we define the Hankel operator $H_{\phi}$ corresponding 
to any function $\phi\in L^\infty(\T^n)$ and 
we give a possible generalization of Nehari's theorem.  
Finally, in Section \ref{sec 5}, 
we give a new proof of the Kor\'{a}nyi--Puk\'{a}nszky theorem 
using the spectral theorem.

Since the bi-holomorphic automorphism group 
of the unit disc $\mathbb D$ acts transitively on $\mathbb D$, 
it follows that the existence of a solution 
to the CF problem is independent of the constant term in $p.$ 
Hence we assume, without loss of generality, 
throughout this paper that $p(\boldsymbol 0) = 0.$ 

\section{Preliminaries}
In this section, we collect the tools that we use repeatedly in what follows. 
The first of these is a variant of the spectral theorem for a pair of commuting normal operators.  
Fix  a bounded open connected  subset $\Omega\subseteq \mathbb C^n$ and define 
the supremum  norm $\|p\|_{ \infty},$  
of a polynomial  in $n$ - variables taking values in 
some normed linear space $E,$  to be  $\sup_{\VZ\in \Omega} \,\|p(\VZ)\|.$

\begin{defn}[Multiplication Operator]
	For $\phi\in L^{\infty}(\T),$ 
	the {\em multiplication operator} 
	$M_{\phi}:L^2(\T)\to L^2(\T)$ 
	is defined by the point-wise product: $M_{\phi}(f)=\phi f.$
\end{defn}
Since $\phi  f\in L^2(\T)$ 
for any $\phi\in L^{\infty}(\T)$ and 
$f\in L^2(\T),$ 
therefore $M_{\phi}$ is well defined 
for all $\phi\in L^{\infty}(\T)$. 
Also $\|M_{\phi}\|=\|\phi\|_{\infty}$ (cf. Theorem 13.14 in \cite{NY}).

\begin{thm}\label{CaseBidisc}
If the power series 
	$\sum_{\boldsymbol \alpha \in \mathbb{N}_{0}^{2}}a_{\boldsymbol\alpha}\VZ^{\boldsymbol\alpha}$ 
	represents a holomorphic function $f$ 
	on the bidisc $\mathbb D^2$ 
	then $|f(\VZ)|\leq 1$ for all $\VZ\in \mathbb{D}^2$ 
	if and only if the operator norm of  
	\[\left( 
	\begin{array}{ccccc}
	  & \vdots & \vdots & \vdots & \\
	 \cdots & M_{p_0} & M_{p_1} & M_{p_2} & \cdots \\
	\cdots & 0 & M_{p_0} & M_{p_1} & \cdots \\
	\cdots & 0 & 0 & M_{p_0} & \cdots \\
		& \vdots & \vdots & \vdots & \\
	\end{array}\right)\]
is at most 1, where $p_n(z)=\sum\limits_{\substack{k=0}}^{n}a_{n-k,k}z^{k}$ 
is a polynomial of degree $n$ in one variable.
\end{thm}
\begin{proof}
Suppose $B^*$ denotes the adjoint of the bilateral shift 
on $\ell^2(\Z)$ and $f$ is a holomorhic map on $\D^2$ 
given by the power series 
$\sum_{\boldsymbol\alpha \in \mathbb{N}_{0}^{2}}a_{\boldsymbol\alpha}\boldsymbol{z}^{\boldsymbol\alpha}.$ 
The joint spectrum of $I\otimes B^*$ and $B^*\otimes B^*$ is $\T^2.$ 
By spectral theorem we know that spectrum of 
$f(I\otimes B^*,B^*\otimes B^*)$ is same as $f(\T^2).$ 
Therefore, by maximum modulus principle,  
\[\|f\|_{\D^2,\infty}=\|f(I\otimes B^*,B^*\otimes B^*)\|=\|M_{p_0}\otimes I+M_{p_1}\otimes B^*+M_{p_2}\otimes B{^*}^2+\cdots\|,\]
where $p_n(z)=\sum\limits_{\substack{k=0}}^{n}a_{n-k,k}z^{k}$ for each $n\in \N.$
\end{proof}
\begin{rem}\label{OneVariable}
 We state separately the special case of this theorem in one variable, namely, a holomorphic map $f$ defined on the unit disc $\mathbb D,$  maps it into $\mathbb D$ if and only if the multiplication operator $M_f$ maps $L^2(\T)$ to $L^2(\T)$ contractively. \end{rem}

\begin{thm}[Parrott's Theorem, \cite{SP}]
	For $i=1,2$, let $\h_i,\,\mathbb{K}_i$ be Hilbert spaces and 
	$\h=\h_1\oplus \h_2,\,\mathbb{K}=\mathbb{K}_1\oplus \mathbb{K}_2.$ 
	If 
	\[
	\left(
	\begin{smallmatrix}
		A\\
		\\
		C\\
	\end{smallmatrix}
	\right):\h_1\to \mathbb{K} \mbox{ \rm and } \big(\begin{matrix} C &  D \end{matrix}\big  ):\h\to \mathbb{K}_2
	\]
	are contractions, 
	then there exists $X\in\mathcal{B}(\h_2,\mathbb{K}_1)$ 
	such that 
	$\left(
	\begin{smallmatrix}
		A & X\\
		C & D\\	
	\end{smallmatrix}
	\right):\h\to \mathbb{K}
	$ is a contraction.
\end{thm}
In this theorem, all the possible choices for $X$ 
are of the form
\[(I-ZZ^*)^{1/2}V(I-Y^*Y)^{1/2}-ZS^*Y,\]
where $V$ is an arbitrary contraction and $Y,\,Z$ 
are determined from the formulae:
\[D=(I-CC^*)^{1/2}Y,\,\, A=Z(I-C^*C)^{1/2}.\]

We recall a very useful criterion, due to Douglas, Muhly and Pearcy (\cite [Prop. 2.2]{DMP}), for contractivity.
\begin{prop}[Douglas--Muhly--Pearcy]  \label{DMP - Prop}
For $i=1,2,$ let $T_i$ be a contraction on a Hilbert space $\mathcal{H}_i$ and let $X$ be an operator mapping $\mathcal{H}_2$ into $\mathcal{H}_1$.  A necessary and sufficient condition that the operator on $\mathcal{H}_1 \oplus \mathcal{H}_2$ defined by the matrix $\left(\begin{smallmatrix}T_1 & X\\
0 & T_2\end{smallmatrix}\right)$ be a contraction is that there exist a contraction $C$ mapping $\mathcal{H}_2$ into $\mathcal{H}_1$ such that
$$X=\sqrt{1_{\mathcal{H}_1} - T_1 T_1^*}~C~\sqrt{1_{\mathcal{H}_2} - T_2^* T_2}.$$
\end{prop}

Let $H^2(\T)$ denote the Hardy space, 
a closed subspace of $L^2(\T)$. 
Let $P_{-}$ denote the orthogonal projection 
of $L^2(\T)$ onto $L^2(\T)\ominus H^2(\T).$

\begin{defn}[Hankel Operator]
	Suppose $\phi$ is an element in $L^{\infty}(\T)$. Then the   
	{\em Hankel operator} $H_{\phi}$ corresponding to the function $\phi$ 
	is the operator $P_{-}\circ M_{\phi}\vert_{H^2(\T)}$. 
\end{defn}

Finally, we recall the well-known theorem due to Nehari relating  the quotient norm to that of the norm of a Hankel operator.  
\begin{thm}[Nehari's Theorem \cite{ZN}]\label{NT1Var}
	Suppose $\phi$ is a function in $L^{\infty}(\T)$ and 
	$H_{\phi}$ is the corresponding Hankel operator. 
	Then
	$\inf\left\{\|\phi-g\|_{\T,\infty}:g\in H^{\infty}(\T)\right\}
	=\|H_{\phi}\|_{op}.$
\end{thm}

\section{The Carath\'{e}odory--Fej\'{e}r Interpolation Problems}\label{CFP}
There are several different known solutions to 
the CF problem when $n=1,$ see (cf. \cite[Page 179]{Nik}). 
For $n>1,$  see \cite{BLTT}  and \cite[Chapter 3]{BMW}) 
for a comprehensive survey of recent results.  
In this article, we shall obtain necessary condition for the existence of a solution to the CF problem  
for $n\in \mathbb N$ 
and an algorithm to construct a solution if one exists.
\subsection{The planar case}
Although, we state the problem below for polynomials $p$ of degree $2$ with $p(0)=0,$ 
our methods apply to the general case.

\begin{cfpbl}[1,2] \label{planar extension}
	Fix a polynomial $p$ of the form $p(z)=a_1z+a_2z^2.$  
	Find a necessary and sufficient condition 
	for the existence of a holomorphic function $g$ 
	defined on the unit disc $\mathbb D$  
	with $g^{(k)}(0) = 0,\, k=0,1,2,$ 
	such that  $\|p + g \|_{\D,\infty}\leq 1.$ 

\end{cfpbl}
\textbf{Solution:} If this problem has a solution, then using Remark \ref{OneVariable}, it can be deduced that 
\[A_2:=\left(
{
\begin{array}{cc}
	a_1 & a_2\\
	0 & a_1\\
\end{array}}
\right)
\]
is a contraction. 
Thus $\|A_2\|\leq 1$ is a necessary condition. 
On the other hand, assuming that $\|A_2\|\leq 1,$ the  existence of $a_3\in\mathbb{C}$ 
such that 
\[A_3:=\left(
{
\begin{array}{ccc}
	a_1 & a_2 & a_3\\
	0 & a_1 & a_2\\
	0 & 0 & a_1\\
\end{array}}
\right)
\]
has operator norm less than or equal to $1$ follows from Parrott's theorem. 
Repeated use of Parrott's theorem generates a sequence $a_3,a_4,\ldots $ such that $\|M_f\| \leq 1,$ where $f(z) = a_1z + a_2 z^2 + \cdots, $ provided the necessary condition $\|A_2\|\leq 1$ is met. Thus $\|A_2\|\leq 1$ is a necessary and sufficient condition
for the existence of a solution to the {\tt CF problem (1,2)}. 

\subsection{Carath\'{e}odory--Fej\'{e}r interpolation problem in two variables}
In the paper \cite{BLTT}, the Carath\'{e}odory--Fej\'{e}r interpolation problem for the polydisc is treated. 
In the case of two variables, a necessary and sufficient condition for existence of a solution is given, see 
Theorem 5.1 of \cite{BLTT}. Also, a slightly different necessary and sufficient condition, again  
for $\mathbb D^2,$ appears in \cite[Theorem 1]{EPP}.  They discuss separately the case $n = 2$ 
and say that it is special due to the dilation theorem for commuting contractions of Ando. 
Our investigations, giving somewhat different necessary and sufficient conditions, not surprisingly, 
is also special in the case of $n=2.$ We therefore discuss this case first.

\begin{cfpbl}[2,2]\label{two Variables}
Let $p$ in $\mathbb{C}[Z_1,Z_2]$ be a fixed but arbitrary polynomial of the form 
	\[p(z_1,z_2)=
	a_{10}z_1+a_{01}z_2+
	a_{20}z_{1}^{2}+a_{11}z_1z_2+a_{02}z_{2}^2.\]
	Find necessary and sufficient conditions 
	for the existence of a holomorphic function 
	$q$ on $\mathbb{D}^2$ with $\big ( \partial_1^{i_1}  \partial_2^{i_2} q\big ) (\boldsymbol 0) = 0,$ 
	 $i_1 + i_2 \leq 2,$ 
	such that $\|p+q\|_{\D^2,\infty}\leq 1$. 
\end{cfpbl}
The theorem given below follows from Theorem \ref{CaseBidisc} and Proposition \ref{DMP - Prop} (Douglas--Muhly--Pearcy).
As in \eqref{1p1p2}, set $p_1(z) =a_{10} +a_{01}z$ and $p_2(z) = a_{20}+a_{11}z+a_{02}z^2.$
\begin{thm}\label{finalCF}
	If $p$ is any complex valued polynomial in two variables 
	of degree at most $2$ with $p(\boldsymbol 0)=0,$ 
	then $|p_1|^2+|p_2|\leq 1$
	is a necessary condition 
	for the existence of a holomorphic function 
	$q:\mathbb D^2 \to \mathbb C,$ 
	with $q^{(I)}(\boldsymbol 0)=0,$ $|I| \leq 2,$ 
	such that  $\|p+q\|_{\mathbb D^2,\infty} \leq  1.$
\end{thm}

\begin{proof}
Suppose $p$ is a complex valued polynomial in two variables 
of degree at most $2$ such that $p(\boldsymbol 0)=0$ 
and there exists a holomorphic function 
$q:\mathbb D^2 \to \mathbb C,$ 
with $q^{(I)}(\boldsymbol 0)=0,$ $|I| \leq 2,$ 
with $\|p+q\|_{\mathbb D^2,\infty} \leq  1.$ 
Then from Theorem \ref{CaseBidisc}, 
we get  
\[\left\|\left(\begin{array}{cc}
M_{p_1} & M_{p_2}\\
0 & M_{p_1}\\
\end{array}\right)\right\|\leq 1.\]
The contractivity criterion of Proposition \ref{DMP - Prop} then implies that 
$|p_1|^2+|p_2|\leq 1.$
\end{proof}
 
Combining  Theorem \ref{CaseBidisc} and Theorem \ref{finalCF}, we obtain the following theorem, which is  the {\tt CF problem (R)} with $n=2,\,d=2.$ 
\begin{thm}\label{Final form}
	For any polynomial $p$ of the form 
	$$p(\boldsymbol  z) = a_{10} z_1 + a_{01}z_2 + a_{20}z_1^2 + a_{11} z_1z_2 + a_{02} z_2^2,$$ 
	there exists a holomorphic function $q,$ 
	defined on the bidisc $\D^2,$ 
	with $q^{(I)}(\boldsymbol 0)=0$ for $|I|=0,1,2,$ 
	such that 
	$$\|p+q\|_{\mathbb D^2, \infty} \leq 1$$ 
	if and only if
	 	$|p_2|\leq 1-|p_1|^2$ 
    and there exists polynomials $p_k$ of degree less or equal to $k$ such that 
	$f:\mathbb{D}\to \mathcal B(L^2(\mathbb{T})),$ where 
	$$\frac{f^{(k)}(0)}{k!}=M_{p_k}\mbox{ for all }k\geq 0,\, p_0 = 0,$$
	defines a holomorphic function with $\sup_{z\in \mathbb D} \|f(z)\|\leq 1.$  
	\end{thm}
	Thus the {\tt CF problem (2,2)} 
	has been reduced to a one variable problem 
	except it now involves holomorphic functions 
	taking values in $\mathcal B(L^2(\mathbb{T})).$ 
	To discuss this variant of the CF problem, we first introduce a very useful notation.

	Let $\h$ be a separable Hilbert space. 
Given a set of $n$ operators 
$A_1,\ldots,A_n$ in $\mathcal{B}(\h),$ 
define the operator 
\[\mathscr{T}(A_1,\ldots,A_n):=\left(
{\begin{array}{ccccc}
	A_1 & A_2 & A_3 &\cdots & A_n \\
	0 & A_1 & A_2 & \cdots & A_{n-1}\\ 
	0 & 0 & A_1 & \cdots & A_{n-2}\\
	\vdots & \vdots & \vdots & \ddots & \vdots \\
	0 & 0 & 0 & \cdots & A_1 \\
\end{array} } 
\right),\]
which is in $\mathcal B(\mathbb H\otimes \mathbb C^n).$
\begin{defn}[Completely Polynomially Extendible]
	Suppose $k\in\N$ and $\{p_j\}_{j=1}^{k}$ 
	is a sequence of polynomials 
	with $\deg(p_j)\leq j$ for all $j=1,\ldots,k$. 
	Then the operator $\mathscr{T}(M_{p_1},\ldots,M_{p_k})$ will be called 
	$m-${\emph{polynomially}} {\emph extendible} 
	if $\|\mathscr{T}(M_{p_1},\ldots,M_{p_k})\|\leq 1$ 
	and there exists a sequence of polynomials 
	$\{p_{l}\}_{l=k+1}^{m}$, 
	with $\deg(p_{l})\leq l$,  
	such that $\|\mathscr{T}(M_{p_1},\ldots,M_{p_m})\|
	\leq 1$. 
	Also, the operator $\mathscr{T}(M_{p_1},\ldots,M_{p_k})$ 
	will be called {\emph{completely polynomially extendible}} 
	if $\mathscr{T}(M_{p_1},\ldots,M_{p_k})$ is 
	$m-$polynomially extendible for all $n\in\N$.
\end{defn}

For $p_1,p_2\in\mathbb{C}[Z],$ polynomials of degree 
at most $1$ and $2$ respectively,  
let $P$ denote the polynomial
$P(z)=M_{p_1}z + M_{p_2}z^2.$ 
We shall say that $P$ is a polynomial 
in the CF class if  there is a holomorphic function 
$f:\mathbb{D}\to \mathcal B(L^2(\mathbb{T}))$ 
satisfying properties stated in Theorem \ref{Final form}. 
Such a function $f$ will be called a 
CF-extension of the polynomial $P$. 
It follows that a solution to 
the {\tt CF problem (2,2)}  exists 
if and only if the polynomial $P$ is in the CF class. 
We have therefore proved the following theorem.

\begin{thm} \label{soln in 2 var}
A  solution to the {\tt CF problem (2,2)} exists 
if and only if the corresponding one variable 
operator valued polynomial $P$ is in the CF class, 
or equivalently, the operator 
	$\mathscr{T}(M_{p_1},M_{p_2})$ 
	is completely polynomially extendible.
\end{thm}

\subsection{Algorithm for finding a solution to the CF problem} Now, we have all the tools to produce an algorithm
for finding all the polynomials $P(z)=M_{p_1}z + M_{p_2}z^2,$ which are in the CF class:

\begin{itemize}
\item If $\|\mathscr{T}(M_{p_1},M_{p_2})\|\leq 1,$ that is,  
if $|p_2|\leq 1-|p_1|^2,$ 
then move to the next step, 
otherwise $P$ is not a CF class polynomial.
\item Parrott's theorem gives all possible operators 
$T\in\mathcal{B}(L^2(\mathbb{T}))$ 
such that $\mathscr{T}(M_{p_1},M_{p_2},T)$ is a contraction.
Let $\mathcal{C}_3$ be the set of all operators $T,$ 
which are multiplication by a polynomial of degree at most $3$ 
and $\mathscr{T}(M_{p_1},M_{p_2},T)$ is a contraction.
If $\mathcal{C}_3$ is empty then $P$ is not a CF class polynomial.
\item For each $k>3,$ using Parrott's theorem we can construct $\mathcal{C}_k,$ 
the set of all operators $T,$ 
which are multiplication by a polynomial of degree at most $k$ 
and $\mathscr{T}(M_{p_1},\ldots, M_{p_{k-1}},T)$ is a contraction, 
where $M_{p_j}$ is an element of $\mathcal{C}_j$ for $j=3,\ldots,k-1.$ 
\item If all of the sets $\mathcal{C}_k$ are non-empty, then and only then  $P$ is a CF class polynomial.
\end{itemize}

It is clear, from Theorem \ref{Final form}, 
that $|p_1|^2+|p_2|\leq 1$ is a necessary condition 
for the existence of a solution to 
the {\tt CF problem (2,2)}. 
This condition, via Parrott's theorem,  
is also equivalent to the condition 
$\|\mathscr T(M_{p_1}, M_{p_2})\| \leq 1.$ 
	We now give some instances, 
	where this necessary condition 
	is also sufficient for the existence of 
	a solution to the {\tt CF problem (2,2)}. 
	This amounts to finding conditions for $\mathscr{T}(M_{p_1},M_{p_2})$ 
	to be completely polynomially extendible. 
	
\begin{thm} 
	Let $p_1(z)=\gamma + \delta z$ and 
	$p_2(z)=(\alpha + \beta z)(\gamma + \delta z)$ 
	for some choice of complex numbers 
	$\alpha,~\beta,~\gamma$ and $\delta.$ 
	Assume that 
	$|p_1|^2+|p_2|\leq 1.$ 
	If either $\alpha \beta \gamma \delta =0$ or 
	$\arg(\alpha)-\arg(\beta)=\arg(\gamma)-\arg(\delta),$ 
	then $\mathscr{T}(M_{p_1},M_{p_2})$ is 
	completely polynomially extendible.
\end{thm}
\begin{proof} 
All through this proof, for brevity of notation, 
for any holomorphic function  $f:\mathbb D \to \mathcal B(L^2(\mathbb T)),$  
we will let $\|f\|$ denote the norm $\sup\{\|f(z)\|_{\rm op}:z\in\mathbb D\}.$  

\vspace*{0.2cm}

\textbf{Case 1:} Suppose the coefficient $\beta$ is $0$. 
Then the polynomial $P$ is of the form 
$$P(z)=M_{p_1}(z + \alpha z^2 ).$$
Define a polynomial $p$ in one variable by the rule $p(z)=z + \alpha z^2$. 
Using Nehari's theorem, we extend the polynomial $p$ to the holomorphic function 
$\tilde{p}(z)=z + \alpha z^2 + \alpha_3 z^3 +\cdots$ 
with the property that 
$\|\tilde{p}\|_{\D,\infty}=\left\|\mathscr{T}(1,\alpha)\right\|.$
Define a function $f:\mathbb D \to \mathcal B(L^2(\mathbb T))$ as the following: 
\[f(z)=M_{p_1}\tilde{p}(z)=M_{p_1}z + M_{p_2}z^2 + M_{p_3}z^3 +\cdots, \]
where $p_k$ is the polynomial $\alpha_k p_1$. 
Also, the norm of the holomorphic map $f$ over the unit disc $\mathbb{D}$ is 
\[\|f\|=\sup_{z\in\D}\|M_{p_1}\tilde{p}(z)\|
=\|M_{p_1}\|\sup_{z\in\D}|\tilde{p}(z)|=\|M_{p_1}\| \|\mathscr{T}(1,\alpha)\|.\]
Thus 
$\|f\|=\|M_{p_1}\otimes \mathscr{T}(1,\alpha)\|=\|\mathscr{T}(M_{p_1},M_{p_2})\|\leq 1.$
Hence, the map $f$ is a required CF-extension of the given polynomial $P$.

\vspace*{0.2cm}

\textbf{Case 2:} Suppose the coefficient  $\alpha$ is $0$. 
Then the polynomial $P$ is of the form  
$$P(z)=M_{p_1}(z + \beta M_z z^2).$$
Define a operator valued function $Q$ on the unit disc $\mathbb{D}$ as 
$Q(z)=z + \beta M_z z^2$ and define a polynomial $r$ on the bidisc $\mathbb{D}^2$ as  
$r(z_1,z_2)=z_1(1 + \beta z_2)$. 
Let $s(z_2)=1+\beta z_2$. 
Suppose
$$\tilde{s}(z_2)=s(z_2) + \beta_2 z^{2}_{2} + \beta_3z_{2}^{3} + \cdots$$
be such that 
$\|\tilde{s}\|_{\D,\infty}=\|\mathscr{T}(1,\beta)\|$. 
If $\tilde{r}:=z_1\tilde{s}(z_2),$ then
$\|\tilde{r}\|=\|\tilde{s}\|=\|\mathscr{T}(1,\beta)\|.$
If 
$$\tilde{Q}(z)=z + M_{\beta z}z^2 + M_{\beta_2 z^2}z^2 + 
\cdots$$ 
and 
$f(z)=M_{p_1}\tilde{Q}(z),$
then
$\|f\|=\|M_{p_1}\tilde{Q}\|\leq \|M_{p_1}\|\|\tilde{Q}\|.$ 
Since $\tilde{s}(M_z)=\tilde{Q}(z)$, 
from the von Neumann inequality, it follows that 
$\|\tilde{Q}\|\leq \|\tilde{s}\|.$ 
Therefore,
$\|f\|\leq \|M_{p_1}\|\|\mathscr{T}(1,\beta)\|
=\|\mathscr{T}(M_{p_1},\beta M_{p_1})\|.$
Hence, we have 
\[\|f\|
\leq 
\left\|\left(
\begin{array}{cc}
	M_z & 0\\
	0 & I\\
\end{array}
\right)
\left(
\begin{array}{cc}
	M_{p_1} & \beta M_{p_1}\\
	0 & M_{p_1}\\
\end{array}
\right)
\left(
\begin{array}{cc}
	M_{z}^{*} & 0\\
	0 & I\\
\end{array}
\right)\right\|
=
\|\mathscr{T}(M_{p_1},M_{p_2})\|
\leq 1. \]
Therefore the function $f$ is a CF-extension of the given polynomial $P.$

\vspace*{0.2cm}

\textbf{Case 3:} Suppose the coefficients $\alpha$ is not $0$ and  
$\beta$ is $0$. 
Then $P(z)=M_{p_1}\left(z + M_{\alpha + \beta z}z^2\right)$. 
Let $Q(z):=z + M_{\alpha + \beta z}z^2$. 
Define
$r(z_1,z_2):=z_1 + \alpha z^{2}_{1} + \beta z_1z_2=
z_1\left(1 + \alpha z_1 + \beta z_2\right).$
Let $\lambda:=|\alpha|/|\beta|$ and 
$a:=\lambda/(1+\lambda)$. 
Define a polynomial $s$ in two variables by the rule 
$$s(z_1,z_2):=1 + \alpha z_1 + \beta z_2=
\left(a + \alpha z_1\right) + \left(1-a + \beta z_2\right).$$
If $h_1(z_1):= a + \alpha z_1$ and 
$h_2(z_2):= 1-a + \beta z_2,$ then 
there exist
$\tilde{h}_1=a + \alpha z_1 + \alpha_2 z_{1}^{2} +\cdots$ 
and $\tilde{h}_2=1-a + \beta z_2 + \beta_2 z_{2}^{2} +\cdots$
with
$\big\|\tilde{h}_1\big\|=\big \|\mathscr{T}(a,\alpha)\big \| 
\mbox{ and } 
\big\|\tilde{h}_2\big\|=\left \|\mathscr{T}(1-a,\beta)\right\|.$
If
\[\tilde{r}(z_1,z_2):=z_1\left(\tilde{h}_1(z_1) + \tilde{h}_2(z_2)\right)=
z_1 + \alpha z_{1}^{2} + \beta z_1z_2 + 
\alpha_2 z_{1}^{3} + \beta_2z_1z_{2}^{2} +\cdots,\]
then $\|\tilde{r}\|\leq \|\tilde{h}_1\| + \|\tilde{h}_2\|$. 
Define a operator valued holomorphic map $Q$ as the following: 
$$\tilde{Q}(z)=Iz + M_{\alpha + \beta z}z^2 + M_{\alpha_2 + \beta_2 z^2}z^3 +\cdots$$
and 
$f(z)=M_{p_1}\tilde{Q}(z)$ $=\sum_{j}M_{p_j}z^j,$
where 
$p_{k+1}(z)=\big(\alpha_k + \beta_kz^k\big)p_1$ 
for all $k>1$. 
Thus, we see that $\|f\|\leq \|M_{p_1}\|\|\tilde{Q}\|.$ 
Since $\tilde{Q}(z)=\tilde{r}(z,M_z),$ 
from the von Neumann inequality, 
it follows that 
\[\|f\|\leq \|M_{p_1}\|\|\tilde{r}\|\leq \|M_{p_1}\|
\left(\|\tilde{h}_1\| + \|\tilde{h}_2\|\right).\]
As $\mathscr{T}(a,|\alpha|)=\lambda \mathscr{T}(1-a,|\beta|)$, 
therefore
$\left(\|\tilde{h}_1\| + \|\tilde{h}_2\|\right)=
\|\mathscr{T}\left(1,|\alpha|+|\beta|\right)\|$
and hence
\[\|f\|\leq \|M_{p_1}\|\|\mathscr{T}\left(1,|\alpha|+|\beta|\right)\|=
\left\|\mathscr{T}\left(\|p_1\|,(|\alpha|+|\beta|)\|p_1\|\right)\right\|.\]
\textbf{subcase 1:} 
Suppose $\gamma\neq 0$, $\delta \neq 0$ 
and 
$\mbox{ arg }(\alpha) - \mbox{ arg }(\beta) = \mbox{ arg }(\gamma) - \mbox{ arg }(\delta)$. 
Then
\[(|\alpha|+|\beta|)\|p_1\|= \|(\alpha + \beta)p_1\|=\|p_2\|.\]
Our hypothesis clearly implies that
$\|p_2\|_\infty + \|p_1\|_\infty^2 \leq 1.$
Hence the norm of the holomorphic map $f$ on the unit disc $\mathbb{D}$ is at most $1$.\\
\textbf{subcase 2:} 
Suppose the quantities $\gamma$ and $\delta$ are both equal to $0$. 
Then we have 
\[(|\alpha|+|\beta|)\|p_1\|= \|(\alpha + \beta)p_1\|=\|p_2\|.\]
As in subcase 1, 
here also $\|f\|\leq 1$ can be inferred easily. 
\end{proof}
\begin{rem}
In the {\tt CF problem (2,2)}, 
if $p_1\equiv 0$ or $p_2\equiv 0,$ 
and $\|\mathscr T(M_{p_1}, M_{p_2})\| \leq 1,$ 
then $\|P\|\leq 1$ and hence $f$ in Theorem \ref{Final form} 
can be taken to be $P$ itself. 	
\end{rem}

Having verified that the necessary condition 
$\|\mathscr{T}(M_{p_1},M_{p_2})\|\leq 1$ is also  sufficient 
for $P$ to be in the CF class in several cases, 
we expected it to be sufficient in general.  
But unfortunately this is not the case. 
We give an example of a polynomial $P$ 
for which $\|\mathscr{T}(M_{p_1},M_{p_2})\|\leq 1$ 
but $P$ is not in the CF class. 

\textbf{Example:} Let $p_1(z)=1/\sqrt{2}$ and 
$p_2(z)=z^2/2$. 
We show that $\mathscr{T}(M_{p_1},M_{p_2})$ 
is not even $3-$polynomially extendible.

It can be seen that 
$\|\mathscr{T}(M_{p_1},M_{p_2})\|\leq 1$. 
Now suppose there exists a polynomial 
$p_3$ of degree at most $3$ such that  
$\|\mathscr{T}(M_{p_1},M_{p_2},M_{p_3})\|\leq 1.$ 
Then Parrott's theorem 
guarantees the existence of a contraction 
$V\in\mathcal{B}\big(L^2(\T)\big)$ 
such that
\[M_{p_3}=\left(I-M_{|p_1|^2}-M_{p_2}\left(I-M_{|p_1|^2}\right)^{-1}M_{p_2}^{*}\right)V-
M_{p_2}\left(I-M_{|p_1|^2}\right)^{-\frac{1}{2}}M_{p_1}^{*}\left(I-M_{|p_1|^2}\right)^{-\frac{1}{2}}M_{p_2}.\]
As we have $(1-|p_1|^2)^{2}-|p_2|^2\equiv 0,$ 
therefore
\[p_3=\frac{-p_{2}^{2}\overline{p_1}}{1-|p_1|^2}=-\frac{z^4}{2\sqrt{2}}.\]
Thus $p_3$ is a polynomial of degree more than $3,$ 
which is a contradiction. 
Hence $\mathscr{T}(M_{p_1},M_{p_2})$ is not even 
$3-$polynomially extendible.

We close this subsection with an open question: 
Find an explicit strengthening of the inequality 
$\|\mathscr{T}(M_{p_1},M_{p_2})\|\leq 1$ to ensure 
that $P$ is in the CF class?

\subsection{The case of n variables}
First, we  obtain an explicit necessary condition for 
the existence of a solution to the CF problem in the 
case of $n$ - variables, $n\in\mathbb N.$ 
The computations, in this case, are analogous 
to those in the case of two variables but they are  
somewhat cumbersome. Nevertheless, we provide the details. 
Also, an algorithm to determine the set of 
all CF class polynomials in $n$ - variables analogous to that in the case of two variables is given. 

We state below, for a given polynomial in $n$ - variables of  degree $d,$  the \CF problem on the polydisc $\mathbb D^n.$ 

\begin{cfpbl}[$n,d$]
	Fix a polynomial $p\in\mathbb{C}[Z_1,\ldots,Z_n]$ of degree $d$ of the form    
\begin{equation}\label{CFpol}	
	p(z_1,\ldots,z_n)=\sum_{j=1}^{n} a_{e_j}z_j+\cdots+
	\sum_{i_1,\ldots,i_d=1}^{n}a_{e_{i_1}+\cdots+e_{i_d}}z_{i_1}\cdots z_{i_d},
	\end{equation}
where $e_j$ is the row vector of length $n$ 
which has 1 at the $j^{th}$ position and $0$ elsewhere
($e_0$ denotes the zero vector).
	Find necessary and sufficient conditions 
	for the existence of a holomorphic function 
	$q$ defined on the polydisc $\mathbb D^n$ 
	with $q^{(I)}(\boldsymbol 0)=0$ for all $|I|\leq d$
	such that
	$\|p+q\|_{\D^n,\infty}\leq 1.$
\end{cfpbl}

Let $f$ be an analytic function on $\mathbb{D}^n,$ which is represented by the power series
\[f(z_1,\ldots,z_n)=\sum_{k=1}^{\infty}\sum\limits_{\substack{{i_1,\ldots,i_k=1}}}^{n}a_{e_{i_1}+\cdots +e_{i_k}}z_{i_1}\cdots z_{i_k}.\] 
Replace $z_j$ by the operator $I^{\otimes (n-j)} \otimes B^{*^{\otimes j}}$ in the power series representation of $f,$ then 
\[f(I^{\otimes (n-1)} \otimes B^{*},\ldots, B^{*^{\otimes n}} )=\sum_{k=1}^{\infty}A_k\otimes B^{*^{\otimes k}},\]
where 
\begin{eqnarray}\label{AkInBilateral}
A_k=\sum\limits_{\substack{{i_1,\ldots,i_k=1}}}^{n}a_{e_{i_1}+\cdots +e_{i_k}}
\prod_{p=1}^{k} (I^{\otimes (n-i_p)} \otimes B^{*^{\otimes (i_p-1)}}).
\end{eqnarray}
In what follows, it would be convenient to replace the operator $A_k$ by the  multiplication operator on $L^2(\mathbb{T}^{n-1}),$ namely, 
\begin{eqnarray}\label{MultiDimension}
M_{p_k}: L^2(\mathbb{T}^{n-1}) \to L^2(\mathbb{T}^{n-1}),\,\, p_k= \sum\limits_{\substack{{i_1,\ldots,i_k=1}}}^{n}a_{e_{i_1}+\cdots +e_{i_k}}
\prod_{p=1}^{k} {z_{n-i_p+1}z_{n-i_p+2}\cdots z_{n-1}}
\end{eqnarray}
with the understanding that if $i_p=1$ then the monomial ${z_{n-i_p+1}z_{n-i_p+2}\cdots z_{n-1}}$ 
is the constant function $1$. Evidently, $A_k$ is unitarily equivalent to $M_{p_k},$ and therefore this makes no difference. 

Note that the polynomial $p_k$ is from the space $\mathscr{M}^{(\ell)}_{n-1},$ with $\ell=\sum_{p=1}^{k}(i_p-1).$   
Since the maximum value of $\sum_{p=1}^{k}(i_p-1)$ 
is at most $(n-1)k,$ therefore the polynomial $p_k$ is of degree at most $(n-1)k.$ 
\begin{prop}\label{spectralthm}
The function $f$ maps $\mathbb{D}^n$ into $\mathbb{D}$ if and only if 
\[\mathscr T(M{p_1}, M{p_2}, \ldots ) = \left(\begin{array}{ccccc}
 & \vdots & \vdots & \vdots & \\
\cdots & M{p_1} & M{p_2} & M{p_2} & \cdots\\
\cdots & 0 & M{p_1} & M{p_2} & \cdots \\
\cdots & 0 & 0 & M{p_1} & \cdots \\
 & \vdots & \vdots & \vdots & \\
\end{array}\right)\]
is a contraction, where $M_{p_k}$ are as defined in \eqref{MultiDimension}.
\end{prop}
\begin{proof}
Apply the spectral theorem, in the form of functional calculus, for a commuting tuples of normal operators  to the $n$ - tuple $(I^{\otimes (n-1)} \otimes 
B^{*}, \ldots ,  B^{*^{\otimes n}}).$ 
\end{proof}
First, a necessary condition for the existence of a solution 
to the {\tt CF problem (n,d)} is now evident.  
\begin{thm}\label{main}
A solution to  the {\tt CF problem (n,d)} exists only if  	 
	$\mathscr{T}(M_{p_1},\ldots , M_{p_d})$ 
	is a contraction, where the operators $M_{p_1}, \ldots , M_{p_d}$  are defined in \eqref{MultiDimension}.
\end{thm}


Second, claim  that the polynomial $p$ in \eqref{CFpol} and the polynomials $p_1, \ldots, p_d$ in \eqref{MultiDimension} determine each other.  Consequently,  the reformulation of the {\tt CF problem (n,d)} announced in the Introduction follows using Proposition \ref{spectralthm}.  
	
To prove the claim, 	first note that $a_{e_{i_1}+\cdots+e_{i_k}}$ and the monomial $\prod_{p=1}^{k} {z_{n-i_p+1}z_{n-i_p+2}\cdots z_{n-1}}$ are invariant under the permutation of the indices $(i_1,\ldots, i_k).$ 
Therefore, to compute the monomial corresponding to $({i_1},\ldots , {i_k}),$ 
we  assume, without loss of generality, that $1\leq i_1\leq \cdots\leq i_k\leq n.$ 
Now, the monomial corresponding to $({i_1},\ldots , {i_k})$  is 
\begin{eqnarray}\label{MonomialGeneral}
z_{n-i_k+1}\cdots z_{n-i_{k-1}}\cdot z^2_{n-i_{k-1}+1}\cdots z^2_{n-i_{k-2}}\cdots 
z^{k-1}_{n-i_2+1}\cdots z^{k-1}_{n-i_1}\cdot z^k_{n-i_1+1}\cdots z^{k}_{n-1}.
\end{eqnarray}
If all of the multi-indices $i_1, \ldots ,i_k$ are $1,$ then the corresponding monomial in \eqref{MonomialGeneral} is the constant function $1.$ 
Otherwise, there exists $s\geq 1,$ such that $i_1=\cdots=i_{s-1}=1$ and $i_{s}>1.$ In this case, 
the exponent of each of the variables $z_{n-i_q+1},\ldots, z_{n-i_{q-1}},$ 
in \eqref{MonomialGeneral}, is $(k+1-q),$ $s\leq q \leq k,$  
where $i_0$ is assumed to be $1.$ 
Thus if $(i_1,\ldots,i_k)\neq (i_1^\prime,\ldots,i_k^\prime),$ 
 $1\leq i_1\leq \cdots\leq i_k\leq n$ and 
$1\leq i_1^\prime\leq \cdots\leq i_k^\prime\leq n,$ 
then it is clear from \eqref{MonomialGeneral} 
that the monomials corresponding  to $(i_1,\ldots,i_k)$  and $(i_1^\prime,\ldots,i_k^\prime)$  are distinct. 
	
To give a necessary and sufficient condition for the existence of a solution to the {\tt CF problem (n,d)}, we need the notion  of \textit{complete polynomial extendibility} 
for the operator $\mathscr T(M_{p_1}, \ldots , M_{p_d})$. 
For each $k\in\{1,\ldots,d\},$ let $M_{p_k}$ be the operator defined in 
\eqref{MultiDimension}, where $p_k$ is the homogeneous term of degree $k$ in \eqref{CFpol}.

\begin{defn}[Completely Polynomially Extendible]
The operator $\mathscr T(M_{p_1},\ldots,M_{p_d}),$ with $p_j\in\mathscr{M}^{(j)}_{n-1},$ for $j=1,\ldots,d,$ is called {\em $m-$polynomially extendible} if there exist $M_{p_{d+1}},\ldots, M_{p_m}$ with  $p_j\in\mathscr{M}^j_{n-1},$ for $d+1\leq j\leq m$ such that $\mathscr T(M_{p_1},\ldots,M_{p_m})$ is a contraction. The operator $\mathscr T(M_{p_1},\ldots,M_{p_d})$ is called {\em completely polynomially extendible} if 
it is $m-$polynomially extendible for each $m.$
\end{defn}

It is easy to provide a necessary and sufficient condition for the existence 
of a solution for the {\tt CF problem (n,d)} using the notion of complete polynomially extendibility.

\begin{thm}
A solution to the {\tt CF problem (n,d)} exists if and only if $\mathscr T(M_{p_1},\ldots,M_{p_k})$ is completely polynomially extendible.
\end{thm}
We give an algorithm to obtain a solution to the {\tt CF problem (n,d)} 
analogous to the one given in the case of the two variables. 
As in that case, 
for $p_k \in \mathscr M_{n-1}^{(k)},$ $1\leq k \leq d,$
let $P$ be the polynomial $P(z)=M_{p_0}+M_{p_1}z+\cdots + M_{p_d}z^d.$
We shall say that $P$ is a polynomial 
in the CF class if  there is a holomorphic function 
$f:\mathbb{D}\to \mathcal B(L^2(\mathbb{T}^{n-1}))$ 
satisfying properties stated in the {\tt CF problem (R)}. 
Such a function $f$ will be called a 
CF-extension of the polynomial $P$. 
It follows that a solution to 
the {\tt CF problem (n,d)} exists 
if and only if the polynomial $P$ is in the CF class. 
\subsection{Algorithm for finding a solution to the CF problem} This algorithm identifies all polynomials $p$ such that the {\tt CF problem (n,d)} admits a solution.

\begin{itemize}
\item If $\|\mathscr T(M_{p_1},\ldots,M_{p_k})\|\leq 1,$ 
then move to the next step, 
otherwise $P$ is not a CF class polynomial.
\item Parrott's theorem gives all possible operators 
$T$ 
such that $\mathscr T(M_{p_1},\ldots,M_{p_k},T)$ is a contraction.
Let $\mathcal{C}_{k+1}$ be the set of all operators $T$ 
such that $T=M_{p_{k+1}}$ for some $p_{k+1}\in\mathscr{M}^{(k+1)}_{n-1}$   
and $\mathscr T(M_{p_1},\ldots,M_{p_k},T))$ is a contraction.
If $\mathcal{C}_{k+1}$ is empty then $P$ is not a CF class polynomial.
\item For each $s>k+1,$ using Parrott's theorem we can construct $\mathcal{C}_s,$ 
the set of all operators $T,$ 
such that $T=M_{p_s}$ for some 
$p_s\in\mathscr{M}^{(s)}_{n-1}$ 
and $\mathscr{T}(M_{p_1},\ldots,M_{p_{s-1}},T)$ is a contraction, 
where $M_{p_j}$ is an element of $\mathcal{C}_j$ for $j=k+1,\ldots,s-1.$ 
\item If all of the sets $\mathcal{C}_s$ are non-empty, then and only then  $P$ is a CF class polynomial.
\end{itemize}

\section{A generalization of Nehari's Theorem}\label{Nehari}
Let $M$ be a closed subspace of a Hilbert space $\h.$ 
For any point $\phi$ in a Hilbert space $\h$, 
the distance of $M$ from $\phi$ is attained at $P(\phi),$ 
where $P$ is the orthogonal projection of $\h$ onto $M$.  
A deep result, Theorem \ref{NT1Var}, of Nehari shows that the distance of a function 
$\phi$ in $L^{\infty}(\T)$ from 
the closed subspace $H^{\infty}(\T)$ is the norm of the Hankel operator 
$H_\phi$ with symbol $\phi.$ 

\subsection{Nehari's theorem for \texorpdfstring{$L^{\infty}(\mathbb T^2)$}{TEXT}}
In this subsection, 
we shall present a possible multivariate generalization 
of Nehari's theorem for $L^{\infty}(\mathbb T^2).$ 
This generalization is most conveniently stated 
in terms of the D-slice ordering on $\Z^2.$  

For a fixed $k\in\mathbb{Z}$, define 
$P_k:=\left\{\left(x,y\right)|x+y=k \right\}.$ 
The subsets $P_k$ of $\mathbb Z^2$ are disjoint and
$\bigsqcup_{k\in\mathbb{Z}}P_k=\mathbb{Z}^2.$
An order on $\mathbb{Z}^2,$ 
which we call the {\tt D-slice} ordering, 
is  defined below. 
It is obtained from the usual co-lexicographic ordering by rotating it through an angle of $\tfrac{\pi}{4}.$ 

\begin{defn}[D-slice ordering] \label{D-slice}
For $(x_1,y_1)\in P_l$ and $(x_2,y_2)\in P_m,$ 
\begin{enumerate}
	\item if $l=m,$ then the order between $(x_1,y_1)$ and  $(x_2,y_2)$  is determined 
	by the lexicographic ordering 
	on $P_l\subseteq \mathbb Z^2$ and    
	\item  if $l < m$ (resp., if $l > m$), 
	then $(x_1,y_1) < (x_2,y_2)$ (resp., $(x_1,y_1) > (x_2,y_2)$).  
\end{enumerate}
\end{defn}

Define
\[H_1:=\Big\{f:=\!\!\!\!\!\sum\limits_{(m,n)\in A_1}\!\!\!\!\!\!\!\!a_{m,n}z_{1}^{m}z_{2}^{n}|f\in L^{\infty}(\T^2)\Big\},\,
H_2:=\Big\{f:=\!\!\!\!\!\sum\limits_{(m,n)\in A_2}\!\!\!\!\!\!\!\!a_{m,n}z_{1}^{m}z_{2}^{n}|f\in L^{\infty}(\T^2)\Big\},\] 
where $A_1:=\{(m,n)\in \Z^2:m+n\geq 0\}$ and $A_2:=\{(m,n)\in \Z^2:m+n<0\}.$ 
$H_1$ and $H_2$ are two closed and disjoint subspaces 
of $L^{\infty}(\T^2)$ satisfying 
$L^{\infty}(\T^2)=H_1 \oplus H_2$. 
The answer to the following question on 
$L^{\infty}(\mathbb T^2)$ 
would be one possible generalization of the Nehari's theorem. 
Let $\mbox{\rm dist}_{\infty}(\phi,H_1),$ denote the distance of $\phi$ 
	from the subspace $H_1?$

\begin{Qn}
	 	For any $\phi$ in $L^{\infty}(\T^2),$ 
	 	what is $\mbox{\rm dist}_{\infty}(\phi,H_1)?$ 
\end{Qn}
To answer this question, it will be convenient to introduce the notion of a Hankel operator 
with symbol $\phi,$  $\phi \in L^\infty(\mathbb T^2).$

\subsection{The Hankel matrix corresponding to \texorpdfstring{$\phi$}{TEXT}}
Suppose $f$ is a function in $L^{2}(\T^2).$  
Then $f$ can be represented in the following power series
\[f(z_1,z_2)=\sum\limits_{m,n\in\Z}\!\!\!\!\!a_{m,n}z_{1}^{m}z_{2}^{n}
=\sum\limits_{m,n\in A_1}\!\!\!\!\!a_{m,n}z_{1}^{m}z_{2}^{n} + 
\sum\limits_{m,n\in A_2}\!\!\!\!\!a_{m,n}z_{1}^{m}z_{2}^{n}.\]
Suppose $z_2=\lambda z_1$. 
Then
\[f(z_1,\lambda z_1)=\sum\limits_{k\geq 0}\left(\sum\limits_{m+n=k}
\!\!\!\!\!a_{m,n}\lambda^n\right)z_{1}^{k} + 
\sum\limits_{k<0}\left(\sum\limits_{m+n=k}\!\!\!\!\!a_{m,n}\lambda^n\right)z_{1}^{k}.\]
Setting 
$f_{k}(\lambda):=\sum_{m+n=k}a_{m,n}\lambda^n,$ 
we have  
\begin{eqnarray}\label{f(z,cz)}
f(z_1,\lambda z_1)=\sum_{k\in\Z}f_{k}(\lambda)z_{1}^{k}.
\end{eqnarray}
In this way, $L^{2}(\T^2)$ 
is first identified  
with $L^{2}(\T)\otimes L^{2}(\T)$ and then  
a second time with 
$L^{2}(\T)\otimes \ell^2(\Z)$, 
the identifications in both cases  
being isometric. 
For any $\phi\in L^{\infty}(\T^2),$ 
define the multiplication operator 
$M_{\phi}:L^{2}(\T)\otimes \ell^2(\Z)\to L^{2}(\T)\otimes \ell^2(\Z)$ 
as follows 
\[M_{\phi}\left(\sum\limits_{j\in\Z}g_j\otimes e_j\right):=
\sum\limits_{k\in\Z}\left(\sum\limits_{q\in\Z}g_q\phi_{q+k}\right)\otimes e_k,\]
where $\phi_j$ satisfies $\phi(z_1,\lambda z_1)=\sum_{j\in \Z}\phi_j(\lambda)z_1^k.$ 

\begin{lem}\label{Norm of Multiplication Operator}
	For any $\phi\in L^{\infty}(\T^2),$ 
	we have $\|M_{\phi}\|=\|\phi\|_{\T^2,\infty}.$ 
\end{lem}
\begin{proof}
Let $\phi\in L^{\infty}(\T^2)$ 
be an arbitrary element. 
From \eqref{f(z,cz)}, 
it follows that 
$\phi(z,\lambda z)=\sum_{k\in\Z}{\phi}_{k}(\lambda)z^{k}$ 
for some ${\phi}_k$ in $L^\infty(\mathbb T).$ 
The set of vectors $\left\{z^i \otimes e_j:(i,j)\in\Z^2\right\}$ 
is an orthonormal basis 
in $L^{2}(\T)\otimes \ell^2(\Z)$. 
The matrix of the operator $M_{\phi}$ 
with respect to this basis 
and the D-slice ordering on its index set is of the form 
\[\left(
\begin{array}{ccccc}
	 & \vdots & \vdots & \vdots & \\
	\cdots & M_{{\phi}_{-1}} & M_{{\phi}_{0}} & M_{{\phi}_{1}} & \cdots\\
	\cdots & M_{{\phi}_{-2}} & M_{{\phi}_{-1}} & M_{{\phi}_{0}} & \cdots\\
	\cdots & M_{{\phi}_{-3}} & M_{{\phi}_{-2}} & M_{{\phi}_{-1}} & \cdots\\
	 & \vdots & \vdots & \vdots & \\ 
\end{array}
\right).
\] 
We know that 
$\|\phi\|_{\T^2,\infty}=\sup_{\lambda\in\T}\sup_{z\in\T}\left|\sum_{k\in\Z}{\phi}_{k}(\lambda)z^{k}\right|$. 
Thus
\begin{align*}
	\|\phi\|_{\T^2,\infty}&=\sup_{\lambda\in\T}
	\left\|
	\left(
	\begin{array}{ccccc}
		 & \vdots & \vdots & \vdots & \\
		\cdots & {\phi}_{-1}(\lambda) & {\phi}_{0}(\lambda) & {\phi}_{1}(\lambda) & \cdots\\
		\cdots & {\phi}_{-2}(\lambda) & {\phi}_{-1}(\lambda) & {\phi}_{0}(\lambda) & \cdots\\
		\cdots & {\phi}_{-3}(\lambda) & {\phi}_{-2}(\lambda) & {\phi}_{-1}(\lambda) &\cdots\\
		 & \vdots & \vdots & \vdots & \\ 
	\end{array}
	\right)
	\right\| \\
	&=
	\left\|
	\left(
	\begin{array}{ccccc}
		 & \vdots & \vdots & \vdots & \\
		\cdots & M_{{\phi}_{-1}} & M_{{\phi}_{0}} & M_{{\phi}_{1}} & \cdots\\
		\cdots & M_{{\phi}_{-2}} & M_{{\phi}_{-1}} & M_{{\phi}_{0}} & \cdots\\
		\cdots & M_{{\phi}_{-3}} & M_{{\phi}_{-2}} & M_{{\phi}_{-1}} & \cdots\\
		 & \vdots & \vdots & \vdots & \\ 
	\end{array}
	\right)
	\right\|.
\end{align*}
Hence $\|\phi\|_{\T^2,\infty}=\|M_{\phi}\|$ completing the proof.
\end{proof}

The Hilbert space $\ell^{2}(\N_0)$ and the normed linear subspace 
$$\big \{(\ldots,0,x_0,x_1,\ldots):\sum_{i\geq 0}|x_i|^2 < 
\infty \mbox{ with }x_0 \mbox{ at the } 0^{th} 
\mbox{ position}\big \}$$
of $\ell^2(\mathbb Z)$ 
are naturally isometrically isomorphic.
Let $H:=L^2(\T)\otimes \ell^2(\N_0)$. 
The space $H$ is a closed subspace of 
$L^{2}(\T)\otimes \ell^{2}(\Z)$. 
We define Hankel operator 
$H_{\phi}$ with symbol $\phi\in L^{\infty}(\T^2)$ 
to be the operator 
$P_{H^{\perp}}\circ {M_{\phi}}_{|H}$. 
Writing down the matrix for $H_{\phi}$ 
with respect to the bases 
$\{z^i \otimes e_j:i\in\Z,j=0,1,2,\ldots\}$ 
and 
$\{z^i \otimes e_{-j}:i\in\Z,j=1,2,\ldots\}$ 
in the spaces $H$ 
and $H^{\perp}$ respectively, 
we get 
\[H_{\phi}=
\left(
\begin{array}{cccc}
	M_{{\phi}_{-1}} & M_{{\phi}_{-2}} & M_{{\phi}_{-3}} & \cdots\\
	M_{{\phi}_{-2}} & M_{{\phi}_{-3}} & M_{{\phi}_{-4}} & \cdots\\
	M_{{\phi}_{-3}} & M_{{\phi}_{-4}} & M_{{\phi}_{-5}} & \cdots\\
	\vdots & \vdots & \vdots & \\ 
	\end{array}
\right),\, \phi \in L^\infty(\mathbb T^2).
\]
We note that the operator $H_{\phi}$ is the Hankel operator with symbol $\phi$ modulo the sign of the indices, see \cite[page 34]{SA}. However, we also note that it is different from the usual definition of either the big or the small Hankel operator in two variables as defined in \cite[Section 4.4]{BMW}.

\begin{lem}\label{Upper bound for Hankel matrix}
For any $\phi$ in $L^\infty(\mathbb T^2),$ 
we have $\|H_{\phi}\|\leq \mbox{\rm dist}_{\infty}(\phi,H_1)$.
\end{lem}
\begin{proof} 
From the definition of $H_{\phi}$ and 
Lemma \ref{Norm of Multiplication Operator}, 
it can be seen that
\[\big\|H_{\phi}\big\|=\big\|P_{H^{\perp}}\circ {M_{\phi}}_{|H}\big\|
\leq \big\|M_{\phi}\big\| = \big\|\phi\big\|_{\T^2,\infty}.\] 
Thus $\|H_{\phi}\|\leq \|\phi\|_{\T^2,\infty}$. 
From the matrix representation of $H_{\phi},$ 
it is clear that for any 
$g$ in $H_1,$ $H_{\phi-g}=H_{\phi}$. 
Hence 
$\|H_{\phi}\|=\|H_{\phi -g}\|\leq \|\phi-g\|_{\T^2,\infty}$. 
\end{proof}
For 
$n\in\N$, $a_0,a_1,\ldots,a_{n-1}\in\C$ 
and 
$(b_m)_{m\in\N},$ $b_m\in \C,$ 
define the  operator  
\[T_n\left((b_m),a_0,a_1,\ldots,a_{n-1}\right):=
\left(
\begin{array}{cccc}
	a_0 & a_1 & \cdots & a_{n-1}\\
	b_1 & a_0 & \cdots & a_{n-2}\\
	\vdots & \vdots & \ddots & \vdots\\
	b_{n-1} & b_{n-2} & \cdots & a_0\\
	\vdots & \vdots &  & \vdots\\
\end{array}
\right).
\]

\begin{lem}\label{Existence of essentially bounded function}
	Suppose    
	$f_0,f_1,\ldots,f_{n-1}\in L^{\infty}(\T)$ 
	and 
	$(g_m)_{m\in \mathbb N},\,g_m \in L^{\infty}(\T)$ 
	are such that 
	\[\sup_{\lambda\in\T}
	\|T_n\left((g_m(\lambda)),f_0(\lambda),\ldots,f_{n-1}(\lambda)\right)\|\leq 1.\]
	Then there exists $f_n\in L^{\infty}(\T)$ satisfying 	
	\[\sup_{\lambda\in\T}
	\|T_{n+1}\left((g_m(\lambda)),f_0(\lambda),\ldots,f_{n}(\lambda)\right)\|\leq 1.\]
\end{lem}
\begin{proof}
Let \[Q(\lambda)=
\left(\!\!
\begin{array}{ccc}
	f_0(\lambda)\!\!\!&\!\!\cdots\!\!&\!\!\!f_{n-1}(\lambda)\\
\end{array}
\right),R(\lambda)=
\left(
\begin{array}{ccccc}
	f_{n-1}(\lambda)\!\!\! &\!\! \cdots \!\!&\!\!\! f_0(\lambda) & \!\!\!g_1(\lambda) \!\!&\!\! \cdots\\
\end{array}
\right)^{\rm t}
\]
and 
\[S(\lambda)=
\left(
\begin{array}{cccc}
	g_1(\lambda) & f_0(\lambda) & \cdots & f_{n-3}(\lambda)\\
	g_2(\lambda) & g_1(\lambda) & \cdots & f_{n-4}(\lambda)\\
	\vdots & \vdots &\ddots & \vdots\\
	g_{n-1}(\lambda) & g_{n-2}(\lambda) & \cdots & g_1(\lambda)\\
	\vdots & \vdots & & \vdots\\
\end{array}
\right).
\]
All possible choices of $f_n(\lambda),$ 
for which  
$T_{n+1}\left((g_m(\lambda)),f_0(\lambda),\ldots,f_{n}(\lambda)\right)$ 
is a contraction are given, via Parrott's theorem 
(cf. \cite[Chapter 12, page 152]{NY}), 
by the formula
\begin{equation}\label{All solutions in Parrot's theorem}
	f_n(\lambda)=(I-ZZ^*)^{1/2}V(I-Y^*Y)^{1/2}-ZS(\lambda)^*Y, 
\end{equation}
where $V$ is an arbitrary contraction 
and the operators $Y,$ $Z$ are obtained from the 
formulae $R(\lambda)=(I-S(\lambda)S(\lambda)^*)^{1/2}Y$, 
$Q(\lambda)=Z(I-S(\lambda)^*S(\lambda))^{1/2}.$ 

We note that every entry of $I-S(\lambda)^*S(\lambda)$ 
is in $L^{\infty}$ as a function of $\lambda.$ 
Thus all entries in $(I-S(\lambda)^*S(\lambda))^{1/2}$ 
are measurable functions which are essentially bounded. 
Consequently, so are entries of $Z.$ 
A similar assertion can be made for $Y.$
Therefore, choosing $V=0$ 
in equation \eqref{All solutions in Parrot's theorem}, 
we get $f_n$ with the required property. 
In fact, one can choose $V$ to be any contraction 
whose entries are $L^{\infty}$ functions. 
\end{proof}
Let $\h$ be a Hilbert space. 
For any $(T_n)_{n\in\N},$ $T_n\in \mathcal{B}(\h),$ 
define an operator $H(T_1,T_2,\ldots)$ as follows: 
\[H(T_1,T_2,\ldots)=
\left(
\begin{array}{cccc}
	T_1 & T_2 & T_3 & \cdots\\
	T_2 & T_3 & T_4 & \cdots\\
	T_3 & T_4 & T_5 & \cdots\\
	\vdots & \vdots & \vdots & \\ 
\end{array}
\right).
\]
\begin{thm}[Nehari's theorem for $L^{\infty}(\mathbb T^2)$] \label{thm 4.6}
	If $\phi\in L^{\infty}(\T^2),$ 
	then $\|H_{\phi}\|=\mbox{\rm dist}_{\infty}(\phi,H_1)$.
\end{thm}
\begin{proof}
From Lemma \ref{Upper bound for Hankel matrix},  
we know that 
$\|H_{\phi}\|\leq {\rm dist}_{\infty}(\phi,H_1)$. 
Without loss of generality 
we assume that $\|H_{\phi}\|=1$.
Using Lemma \ref{Existence of essentially bounded function}, 
we find 
${\phi}_{0}\in L^{\infty}(\T)$ 
such that the norm of the operator 
$H\big (M_{{\phi}_{0}},M_{{\phi}_{-1}},\ldots \big )$ 
is at most 1. 
Therefore, one proves the desired conclusion by repeated use of 
Lemma \ref{Existence of essentially bounded function}.
\end{proof}

\subsection{Nehari's Theorem in \texorpdfstring{$n$}{TEXT} - variables} 
The generalization of Nehari's theorem for $n$ - variables is very similar to that of the 
case of two variables. Therefore we give the details in this case only briefly. The key 
is  the  {\tt D-slice ordering} on $\mathbb{Z}^n,$ which is defined below. 

For a fixed $k\in\mathbb{Z}$, define 
$P_k:=\left\{\left(x_1,\ldots,x_n\right)\in \mathbb Z^n\, |\, x_1+ \cdots +x_n=k \right\}.$ 
The subsets $P_k$ of $\mathbb Z^n$ are disjoint and 
$\bigsqcup_{k\in\mathbb{Z}}P_k=\mathbb{Z}^n.$
\begin{defn}[D-slice ordering for $\Z^n$] \label{D-slice N var}
For $(x_1,\ldots,x_n)\in P_l$ and $(y_1,\ldots,y_n)\in P_m,$ 
\begin{enumerate}
	\item if $l=m,$ then the order between  $(x_1,\ldots,x_n)$ and $(y_1,\ldots,y_n)$ is determined by the lexicographic ordering on $P_l\subseteq \mathbb Z^n$ and    
	\item if $l < m$ (respectively if $l > m$), then $(x_1,\ldots,x_n) < (y_1,\ldots,y_n)$ (respectively $(x_1,\ldots,x_n) > (y_1,\ldots,y_n)$). 
\end{enumerate}
\end{defn}
Define
\[H_1:=\Big\{f:=\!\!\!\!\!\sum\limits_{(m_1,\ldots,m_n)\in A_1}\!\!\!\!\!\!\!\!a_{m_1,\ldots,m_n}z_{1}^{m_1}\cdots z_{n}^{m_n}|f\in L^{\infty}(\T^2)\Big\} \mbox{ and }\]
\[H_2:=\Big\{f:=\!\!\!\!\!\sum\limits_{(m_1,\ldots,m_n)\in A_2}\!\!\!\!\!\!\!\!a_{m_1,\ldots,m_n}z_{1}^{m_1}\cdots z_{n}^{m_n}|f\in L^{\infty}(\T^2)\Big\},\] 
where $A_1:=\{(m_1,\ldots,m_n)\in \Z^n:m_1+\cdots+m_n\geq 0\}$ and $A_2:=\{(m_1,\ldots,m_n)\in \Z^n:m_1+\cdots+m_n<0\}.$ 
The two subspaces $H_1$ and $H_2$ of $L^{\infty}(\T^n)$ are closed and disjoint, moreover 
$L^{\infty}(\T^n)=H_1 \oplus H_2$. 
Let $f$ be a function in $L^{2}(\T^2)$ and 
\begin{eqnarray*}
f(z_1,\ldots, z_n)&=&\sum\limits_{m_1,\ldots, m_n\in\Z}\!\!\!\!\!a_{m_1,\ldots, m_n}z_{1}^{m_1}\cdots z_{n}^{m_n}\\
&=&\sum\limits_{(m_1,\ldots, m_n)\in A_1}\!\!\!\!\!a_{m_1,\ldots, m_n}z_{1}^{m_1}\cdots z_{n}^{m_n} + 
\sum\limits_{(m_1,\ldots, m_n)\in A_2}\!\!\!\!\!a_{m_1,\ldots, m_n}z_{1}^{m_1}\cdots z_{n}^{m_n}
\end{eqnarray*}
be its power series expansion.
Suppose $z_j=\lambda_{j-1} z_1,$ $\lambda_{j-1}\in\mathbb{D}$ for $j=2,\ldots,n$. 
Then
\begin{eqnarray*}
f(z_1,\lambda_1 z_1,\ldots,\lambda_{n-1}z_1)&=&\sum\limits_{k\geq 0}\left(\sum\limits_{m_1+\cdots + m_n=k}
\!\!\!\!\!a_{m_1,\ldots, m_n}\lambda_{1}^{m_2}\cdots \lambda_{n-1}^{m_n}\right)z_{1}^{k}\\
 &+& 
\sum\limits_{k<0}\left(\sum\limits_{m_1+\cdots + m_n=k}
\!\!\!\!\!a_{m_1,\ldots,m_n}\lambda_{1}^{m_2}\cdots \lambda_{n-1}^{m_n}\right)z_{1}^{k}.
\end{eqnarray*}
For each $k\in\mathbb{Z},$ we set 
$f_{k}(\lambda_1,\ldots,\lambda_{n-1}):=\sum\limits_{m_1+\cdots + m_n=k}a_{m_1,\ldots, m_n}\lambda_{1}^{m_2}\cdots \lambda_{n-1}^{m_n}.$ 
For any $\phi\in L^{\infty}(\T^n),$ 
define the multiplication operator 
$M_{\phi}:L^{2}(\T^{n-1})\otimes \ell^2(\Z)\to L^{2}(\T^{n-1})\otimes \ell^2(\Z),$ 
$$M_{\phi}\Big(\sum\limits_{j\in\Z}g_j\otimes e_j\Big):=
\sum\limits_{k\in\Z}\Big(\sum\limits_{q\in\Z}g_q\phi_{q+k}\Big)\otimes e_k,$$ 
where $\phi_j$ satisfies $\phi(z_1,\lambda_1 z_1,\ldots,\lambda_{n-1}z_1)=\sum_{j\in \Z}\phi_j(\lambda_1,\ldots,\lambda_{n-1})z_1^k.$  
Now, we define the Hankel operator $H_{\phi}$ 
corresponding to the function $\phi$ to be the following operator:
\[H_{\phi}=
\left(
\begin{array}{cccc}
	M_{{\phi}_{-1}} & M_{{\phi}_{-2}} & M_{{\phi}_{-3}} & \cdots\\
	M_{{\phi}_{-2}} & M_{{\phi}_{-3}} & M_{{\phi}_{-4}} & \cdots\\
	M_{{\phi}_{-3}} & M_{{\phi}_{-4}} & M_{{\phi}_{-5}} & \cdots\\
	\vdots & \vdots & \vdots & \\ 
	\end{array}
\right).\]
The proof of the following theorem is very similar  to that of Theorem \ref{thm 4.6}, 
therefore we omit the details.
\begin{thm}[Nehari's theorem for $L^{\infty}(\mathbb T^n)$] 	
	If $\phi\in L^{\infty}(\T^n),$ 
	then $\|H_{\phi}\|=\mbox{\rm dist}_{\infty}(\phi,H_1)$.
\end{thm}

\subsection{CF problem in \texorpdfstring{$\mathbb D^2$}{TEXT} and Nehari's theorem for \texorpdfstring{$L^{\infty}(\T^2)$}{TEXT}}
Fix $p\in\mathbb{C}[Z_1,Z_2]$ to be the polynomial defined by 
	$$p(z_1,z_2)=a_{10}z_1+a_{01}z_2+
	a_{20}z_{1}^{2}+a_{11}z_1z_2+a_{02}z_{2}^2.$$
Denote $\phi(z_1,z_2):=\overline{z}_1^3p(z_1,z_2)=a_{10}\overline{z}_1^2+a_{01}\overline{z}_1^3z_2+a_{20}\overline{z}_1+a_{11}\overline{z}_1^2z_2+a_{02}\overline{z}_1^3z^2_2.$ 
Suppose $p_1(\lambda)=a_{10}+a_{01}\lambda$ 
and $p_2(\lambda)=a_{20}+a_{11}\lambda + a_{02}\lambda^2.$
Then $\|H_{\phi}\|={\rm dist}_{\infty}(\phi,H_1),$ where 
\[H_\phi=
\left(
	\begin{array}{cccc}
		M_{p_2} & M_{p_1}& 0 & \cdots\\
		M_{p_1} & 0 & 0 & \cdots\\
		0 & 0 & 0 & \cdots\\
		\vdots&\vdots&\vdots& \\
	\end{array}
\right)
\]	
Thus, if there exists a holomorphic function 
$q:\D^2\to \C$ with $q^{(I)}(0)=0$ for $|I|\leq 2$ 
such that $\|p+q\|_{\D^2,\infty}\leq 1,$ 
then $\|H_\phi\|\leq \|p+q\|_{\D^2,\infty}$. 
Hence $\|H_\phi\|\leq 1$ is a 
necessary condition for such a $q$ to exist. As we have seen before, this necessary condition, however, is not sufficient.

\section{An alternative proof of the Kor\'{a}nyi--Puk\'{a}nszky Theorem}\label{sec 5}
\label{Koranyi}
We recall the following theorem of Kor\'{a}nyi and Puk\'{a}nszky 
\cite[Corollary, Page 452]{KP}. 
This gives a necessary and sufficient condition 
for the range of a  holomorphic function, defined 
on the polydisc $\mathbb D^n,$ to  
be in the right half plane $H_+.$ 

\begin{thm}[Kor\'{a}nyi--Puk\'{a}nszky Theorem] \label{Koranyi Theorem}
	Suppose the power series 
	$\sum_{\boldsymbol\alpha \in \mathbb{N}_{0}^{n}}a_{\boldsymbol\alpha}\boldsymbol{z}^{\boldsymbol\alpha}$ 
	represents a holomorphic function $f$ 
	on the polydisc $\mathbb D^n,$ then   
	$\Re(f(z))\geq 0$ for all $\VZ\in \mathbb{D}^n$ 
	if and only if the map 
	$\phi:\mathbb{Z}^n\to \mathbb{C}$ 
	defined by 
	\begin{eqnarray*}
		\phi (\boldsymbol\alpha )=\left\{
		\begin{array}{ll}
		      2\Re a_{\boldsymbol\alpha} & \mbox{\rm if }\boldsymbol\alpha =0\\
		      a_{\boldsymbol\alpha} & \mbox{\rm if }\boldsymbol\alpha > 0\\
		      a_{-\boldsymbol\alpha} & \mbox{\rm if }\boldsymbol\alpha < 0\\
		      0 & \mbox{\rm otherwise }\\
		\end{array} 
		\right.
	\end{eqnarray*}
	is positive, that is, the $k\times k$ matrix 
	$\big (\!\!\big ( \phi(\scriptstyle{m_i-m_j}) \big )\!\!\big )$ 
	is non-negative definite 
	for every choice of $k\in \N$ and $m_1, \ldots, m_k\in \mathbb Z^n.$  
	Here $\Re z$ denotes the real part of the complex number $z.$   
\end{thm}
\noindent 
We call this function 
$\phi,$ the Kor\'{a}nyi--Puk\'{a}nszky function 
corresponding to the coefficients 
$(a_{\boldsymbol\alpha})_{\boldsymbol\alpha\in \mathbb{N}_0^n}.$

\subsection{The planar Case}
Suppose $f$ is holomorphic mapping from unit disc $\D$ into $H_+.$ 
Without loss of generality we can assume $f(0)=1/2.$
Consider the Cayley map 
$\chi :H_{+}\to \mathbb{D}$ defined by 
\[\chi(z)=\frac{1-z}{1+z},\] 
which is a bi-holomorphism. 
Suppose $\chi \circ f$ 
has the power series 
$\sum_{n=1}^{\infty}a_nz^n$ 
mapping $\mathbb{D}$ into $\mathbb{D}$. Then
\begin{equation}
 	 f(z) = \frac{1+\chi \circ f(z)}{1-\chi \circ f(z)}
 	=2\left(c_0+\sum_{n=1}^{\infty}c_nz^n\right),  \label{c_n's}
\end{equation}
where $2c_n = f^{(n)}(0)/n!.$ 
The exact relationship between the co-efficients 
$c_n$  and $a_n$ are obtained in the lemma below. 
In this section, we set $c_0=1/2,$ wherever it occurs. 
\begin{lem} \label{recursive coefficients}
	The coefficient $c_n$ in equation \eqref{c_n's} 
	is given by 
	$a_n+\sum\limits_{\substack{j=1}}^{n-1}a_jc_{n-j},$ 
	 $n \geq 1$.
\end{lem}

\begin{proof}
Consider the expression 
\[\chi \circ f(z) =2\left(c_0+\sum_{n=1}^{\infty}c_nz^n\right) 
= 2\left(\frac{1}{2}+f(z)+f(z)^2+f(z)^3+\cdots \right).\]
Rewriting, we get
\[ \frac{1}{1-f(z)} = 1 + \sum\limits_{n = 1}^{\infty}c_nz^n.\]
Hence, we have 
\[\left(1+\sum_{n=1}^{\infty}c_n z^n\right)
\left(1-\sum_{n=1}^{\infty}a_n z^n\right) =1.\]
A comparison of the coefficients completes the verification. 
\end{proof}

Let $\phi$ denote the Kor\'{a}nyi--Puk\'{a}nszky function 
corresponding to the coefficients  
$(c_n)_{n=0}^{\infty}.$

 \textbf{Matrix of $\phi$ :} 
The matrix $\left(\phi(j-k)\right)_{j,k}$  is given by
\begin{equation}
	\bordermatrix{
		~ & \cdots & -1 & 0 & 1 & \cdots\cr
		\vdots & & \vdots & \vdots & \vdots & \cr
		-1 & \cdots & 1 & \overline{c}_1 & \overline{c}_2 & \cdots\cr
		0 & \cdots & c_1 & 1 & \overline{c}_1 & \cdots \cr
		1 & \cdots & c_2 & c_1 & 1 & \cdots \cr
		\vdots &  & \vdots & \vdots & \vdots &  \cr}. \label{Matrix of phi}
\end{equation}
 
For each $n\in \mathbb N,$ let $C_n, A_n$ and $P_n$ denote the  matrices 
\[(C_n:=)\,\, \left(
{
\begin{array}{ccccccc}
	1 & \overline{c}_1 & \overline{c}_2 & \cdots & \overline{c}_n \\
	c_1 & 1 & \overline{c}_1 & \cdots & \overline{c}_{n-1}\\ 
	c_2 & c_1 & 1 & \cdots & \overline{c}_{n-2}\\
	\vdots & \vdots & \vdots & \ddots & \vdots \\
	c_n & c_{n-1} & c_{n-2} & \cdots & 1 \\
\end{array} }
\right),
\,(A_n:=)\,\, \left(
{
\begin{array}{ccccc}
	a_1 & a_2 & a_3 &\cdots & a_n \\
	0 & a_1 & a_2 & \cdots & a_{n-1}\\ 
	0 & 0 & a_1 & \cdots & a_{n-2}\\
	\vdots & \vdots & \vdots & \ddots & \vdots \\
	0 & 0 & 0 & \cdots & a_1 \\
\end{array} }
\right)
\]
and 
\[(P_n:=)\,\, \left(
{
\begin{array}{ccccc}
	1 & -a_1 & -a_2 &\cdots & -a_n \\
	0 & 1 & -a_1 & \cdots & -a_{n-1}\\ 
	\vdots & \vdots & \ddots & \ddots & \vdots \\
	0 & 0 & 0 & \cdots & -a_1 \\
	0 & 0 & 0 & \cdots & 1 \\
\end{array} }
\right)
\]
respectively.
 
\begin{lem}\label{contractivity vs positivity}
	For all $n\in\mathbb{N},$ $P_n C_{n}^{\rm t} P_{n}^{*}=(I-A_n A_{n}^{*})\oplus 1.$ 
	\end{lem}
\begin{proof}
We prove the result by induction on $n$. 
The case $n=1$ is trivial. 
Assume the result is valid for $n-1,$  $n>1$. 
For each $n\in \N,$ let  
\[\tilde{P}_n:=\left(-a_n,-a_{n-1},\ldots ,-a_1\right)^{\rm t} 
\mbox{ \rm and } 
\tilde{C}_n:=\left(c_n,c_{n-1},\ldots ,c_1\right)^{\rm t}.\] 
The verification of the identity   
\[P_nC_n^{\rm t}P_{n}^{*} =\left(
\begin{array}{cc}
	P_{n-1} & \tilde{P}_n\\
	0 & 1\\
\end{array}
\right) 
\left(
\begin{array}{cc}
	C_{n-1}^{\rm t} & \tilde{C}_n\\
	\tilde{C}_{n}^{*} & 1\\
\end{array}
\right) 
\left(
\begin{array}{cc}
	P_{n-1}^{*} & 0\\
	\tilde{P}_{n}^{*} & 1\\
\end{array}
\right)\]
is easy. Hence 
\[P_nC_{n}^{\rm t}P_{n}^{*} = 
\left(
\begin{array}{cc}
	P_{n-1}C_{n-1}^{\rm t}P_{n-1}^{*} + \tilde{P}_n\tilde{C}_{n}^{*}P_{n-1}^{*} + 
	\tilde{P}_{n}^{*}\left(P_{n-1}\tilde{C}_n+
	\tilde{P}_n\right) & P_{n-1}\tilde{C}_n + \tilde{P}_n\\
	\left(P_{n-1}\tilde{C}_n +\tilde{P}_n\right)^* & 1\\
\end{array} 
\right).\]
From Lemma \ref{recursive coefficients}, 
we have the identity 
$P_{n-1}\tilde{C}_n + \tilde{P}_n = 0$ 
and therefore we conclude that
\[P_nC_n^{\rm t}P_{n}^{*}=
\left(
\begin{array}{cc}
	P_{n-1}C_{n-1}^{\rm t}P_{n-1}^{*}+ \tilde{P}_n\tilde{C}_{n}^{*}P_{n-1}^{*} & 0\\
	0 & 1\\
\end{array}
\right).
\]
Now, 
\[\tilde{P}_n \tilde{C}_{n}^{*}P_{n-1}^{*}
=\left(
\begin{array}{c}
	-a_n \\
	\vdots \\
	-a_1\\
\end{array}
\right)
\left(
\begin{array}{cccc}
	\overline{c}_n-\sum\limits_{i=1}^{n-1}a_ic_{n-i}, & 
	\overline{c}_{n-1}-\sum\limits_{i=1}^{n-2}a_ic_{n-i}, & 
	\cdots, & \overline{c}_1\\
\end{array}
\right).
\]
From Lemma \ref{recursive coefficients}, 
we get
\[\tilde{P}_n\tilde{C}^{*}_{n}P_{n-1}^{*}
=\left(
\begin{array}{c}
	-a_n \\
	\vdots \\
	-a_1\\
\end{array}
\right)
\left(
\begin{array}{ccc}
	\overline{a}_n  & \cdots & \overline{a}_1\\
\end{array}
\right)=
\left(-a_{n-i}\overline{a}_{n-j}\right)_{i,j=0}^{n-1}.\]
Since
\[I-A_k A^{*}_{k} = 
\left(
\begin{array}{cccc}
	1-\sum\limits_{j=1}^{k}|a_j|^2 & 
	-\sum\limits_{j=2}^{k}a_j\overline{a}_{j-1} & 
	\cdots & -a_k\overline{a}_1 \\
	-\sum\limits_{j=2}^{k}\overline{a}_ja_{j-1} & 
	1-\sum\limits_{j=1}^{k-1}|a_j|^2 & \cdots & 
	-a_{k-1}\overline{a}_1\\
	\vdots & \vdots & \ddots & \vdots \\
	-a_1\overline{a}_{k} & 
	-a_1\overline{a}_{k-1} & 
	\cdots & 1-|a_1|^2 \\ 
\end{array}
\right),\]
$I-A_nA_{n}^{*} = \left((I-A_{n-1}A_{n-1}^*)\oplus 1\right) + 
\left(-a_{n-j}\overline{a}_{n-l}\right)_{1\leq j,l\leq {k-1}}.$
Thus
$I-A_nA_{n}^{*} = P_{n-1}C_{n-1}^{\rm t}P_{n-1}^{*} + 
\tilde{P}_n\tilde{C}_nP_{n-1}^{*},$
the proof is complete.
\end{proof}
An immediate corollary to 
Lemma \ref{contractivity vs positivity} 
is the following proposition.

\begin{prop}\label{Reason4KP}
	The matrix $C_n$ is non-negative definite 
	if and only if $\|A_n\|\leq 1$.
\end{prop}
Since $\chi \circ f(=g, \mbox{ say})$ is a holomorphic map 
from $\D$ to $\D,$ therefore 
the multiplication operator $M_g$ on $L^2(\T)$ has the property that 
$\|M_g\|=\|g\|_{\D,\infty}$ (see \cite[Theorem 13.14]{NY}). 
Writing the matrix for $M_g$ with respect to the basis 
$\{\ldots,z^{-2},z^{-1},1,z^1,z^2,\ldots\},$ 
we conclude that $M_g$ is a contraction if and only if 
$A_n$ is a contraction for each $n\in \N.$ 
Using Proposition \ref{Reason4KP} together with 
the equality $\|M_g\|=\|g\|_{\D,\infty},$ 
we see that  the function $f$ maps the unit disc $\D$ 
into the right half plane $H_+$ if and only if 
$C_n$ are non-negative definite for each $n\in \N.$ 
Thus we recover the solution to the \CF problem in one variable. 

\subsection{The case of several variables}
In this subsection, all the computations are given for the case of $n=2$ only. 
These computations are easily seen to work equally well, using the {\tt D-slice ordering} on $\mathbb{Z}^n,$  in the case of  an arbitrary $n\in \mathbb N.$ The details are briefly indicated at the end in Subsection \ref{subsec5.3} 

Suppose $f$ is holomorphic mapping from bidisc $\D^2$ into $H_+.$ 
Without loss of generality, we assume that $f(\boldsymbol 0)=1/2.$ As before, let $\chi:\mathbb H \to \mathbb D$ be the Cayley map and 
\[\chi \circ f(\boldsymbol z)=\sum\limits_{\substack{{m,n=0}}}^{\infty}a_{mn}z_{1}^{m}z_{2}^{n}\]
be the power series expansion of the function $\chi \circ f.$
Thus $\chi \circ f$ maps $\mathbb{D}^2$ into $\mathbb{D}$ and $a_{00}=0$. With the understanding that $c_0=1/2,$ 
we have 
\begin{equation*}
 	 f(\boldsymbol z) = \frac{1+\chi \circ f(\boldsymbol z)}{1-\chi \circ f(\boldsymbol z)}
 	=2\left(c_{00}+\sum_{m,n=1}^{\infty}c_{mn}z_{1}^{m}z_{2}^{n}\right).
\end{equation*}
Let $\phi$ be the Kor\'{a}nyi--Puk\'{a}nszky function 
corresponding to the coefficients $(c_{mn}).$ 
The following theorem describes 
the  function $\phi$ 
with respect to the D-slice ordering on $\mathbb Z^2.$
 
\begin{thm}\label{Matrix phi}
Let $(c_{mn})_{m,n\in\N_0}$ be an infinite array of complex numbers. 
The matrix of the  Kor\'{a}nyi--Puk\'{a}nszky function 
$\phi,$ in the D-slice ordering,   
corresponding to this array is of the form   
$$\bordermatrix{
	~ & \cdots & P_{-1} & P_0 & P_1 & \cdots\cr
	\vdots &  & \vdots & \vdots & \vdots & \cr
	P_{-1} & \cdots & I & C_1^{*} & C_2^{*} & \cdots\cr
	P_{0} & \cdots & C_1 & I & C_1^{*} & \cdots \cr
	P_1 & \cdots & C_2 & C_1 & I & \cdots \cr
	\vdots &  & \vdots & \vdots & \vdots &  \cr},$$
where $C_n:=c_{n0}I+c_{n-1,1}B^{*}+\cdots +c_{0n}{B^*}^n,$ 
$n\in \mathbb N$ and $B$ is the bilateral shift on $\ell^2(\Z).$ 
\end{thm} 
\begin{proof}
With respect to the D-slice ordering on $\mathbb{Z}^2,$ 
the matrix corresponding to the function $\phi$ 
is a doubly infinite block matrix.  
The $(k,n)$ element in the $(l,m)$ block in this matrix,  
which is of the form $\phi\left((k,-k+l)-(n,-n+m)\right),$ 
computed below separately: 
\begin{enumerate}
\item[${\bf{k-n<0:}}$	] The quantity $\phi\left((k,-k+l)-(n,-n+m)\right)$ is non-zero 
	only if $k-n \geq l-m.$ 
	Hence if $l\geq m,$ 
	then $\phi\left((k,-k+l)-(n,-n+m)\right)=0.$ 
	Now, assume $l<m.$  
	In this case, the possible values 
	for $k-n$ are $l-m,l-m+1,\ldots,-1,$
	otherwise $\phi\left((k,-k+l)-(n,-n+m)\right) = 0.$ 
	For $p\in\{0,1,\ldots,-l+m-1\}$ and $k-n=l-m+p,$   
	we have
	\[\phi\left((k,-k+l)-(n,-n+m)\right)=\overline{c}_{m-l-p,p}.\]
\item[${\bf{k-n=0:}}$]
		\begin{eqnarray*}
		\phi (0,l-m)=\left\{
		\begin{array}{ll}
 		     c_{0,l-m} & \mbox{ if } l\geq m\\
		      \overline{c}_{0,m-l} & \mbox{ if } l< m\\
	\end{array} 
	\right.
	\end{eqnarray*}
\item[${\bf{k-n>0:}}$] 
	The quantity $\phi\left((k,-k+l)-(n,-n+m)\right)$ 
	is non-zero only if $k-n \leq l-m$. 
	Hence if $l\leq m,$ 
	then $\phi\left((k,-k+l)-(n,-n+m)\right)=0$. 
	Now, assume $l>m$. 
	In this case, the possible values 
	for $k-n$ are $l-m,l-m-1,\ldots,1,$ otherwise 
	$\phi\left((k,-k+l)-(n,-n+m)\right)=0.$ 
	For $p\in\{0,1,\ldots,l-m-1\}$ and $k-n=l-m-p,$ 
	we have
	\[\phi\left((k,-k+l)-(n,-n+m)\right)=c_{l-m-p,p}.\]
\end{enumerate}
Therefore, the $(l,m)$ block $\phi(l,m)$ in the matrix of $\phi$ 
 is of the form  
$$
\phi(l,m) = \begin{cases} C_{m-l}^{*}  &\mbox{if } l < m \\ 
I &  \mbox{if } m=l\\
C_{l-m} & \mbox{if } l>m. \\
\end{cases}
$$
Hence the block matrix of the  Kor\'{a}nyi--Puk\'{a}nszky function 
$\phi,$ in the D-slice ordering, corresponding to 
the array $(c_{mn})$ takes the form 
\[\bordermatrix{
	~ & \cdots & P_{-1} & P_0 & P_1 & \cdots\cr
	\vdots &  & \vdots & \vdots & \vdots & \cr
	P_{-1} & \cdots & I & C_1^{*} & C_2^{*} & \cdots\cr
	P_{0} & \cdots & C_1 & I & C_1^{*} & \cdots \cr
	P_1 & \cdots & C_2 & C_1 & I & \cdots \cr
	\vdots &  & \vdots & \vdots & \vdots &  \cr}.\]
\end{proof}

\begin{lem} \label{recurrence 2 variables}
For all $n\in \mathbb N,$ setting 
$A_n:=a_{n0}I+a_{n-1,1}B^*+\cdots +a_{0n}{B^*}^n,$ $C_n=c_{n0}I+c_{n-1,1}B^*+ \cdots +c_{0n}{B^*}^n,$ 
we have 
		$$C_n=A_n+\sum\limits_{\substack{j=1}}^{n-1}A_jC_{n-j}.$$
\end{lem}
\begin{proof}
Let 
$C(z_1,z_2):=\sum_{i,j=0}^{\infty}c_{ij}z_{1}^{i}z_{2}^{j}.$
We have 
\[1+\chi\circ f(z_1,z_2)+\chi\circ f(z_1,z_2)^2+\cdots=\frac{f(z_1,z_2)}{2}+c_{00}= C(z_1,z_2).\]
Thus 
$C(z_1,z_2)(1-\chi\circ f(z_1,z_2))=1,$ which is the same as  
\begin{eqnarray*} 
	\lefteqn{(1+c_{10}z_1+c_{01}z_2+c_{20}z_{1}^{2}+
	c_{11}z_1z_2+c_{02}z_{2}^{2}+\cdots )\times}\\&\phantom{\times \times}& (1-a_{10}z_1-a_{01}z_2-
	a_{20}z_{1}^{2}-a_{11}z_1z_2-a_{02}z_{2}^{2}+\cdots)=1.\\
\end{eqnarray*}
For each $k\in\mathbb{N},$ comparing the coefficient of the monomial 
$z_{1}^{n-k}z_{2}^k,$ we have
\[c_{n-k,k}=\sum\limits_{\substack{p=0}}^{k}
\sum\limits_{\substack{j=k}}^{n}a_{n-j,p}c_{j-k,k-p},\]
where $a_{00}=0$.\\
The coefficient of ${B^*}^k$ in $A_n+\sum\limits_{\substack{i=1}}^{n-1}A_iC_{n-i}$ 
is 
\begin{eqnarray*}
\lefteqn{a_{n-k,k}c_{00}+a_{n-k,k-1}c_{01}+a_{n-k-1,k}c_{10}+
	a_{n-k,k-2}c_{02}}\\
 &\phantom{++++}& + a_{n-k-1,k-1}c_{11}+a_{n-k-2,k}c_{20}+\cdots\\
\end{eqnarray*} 
\vspace{-40pt}
\begin{eqnarray*}
\lefteqn{= (a_{n-k,k}c_{00}+a_{n-k,k-1}c_{01}+\cdots +
 	a_{n-k,0}c_{0k})+}\\
 	&\phantom{++++}&(a_{n-k-1,k}c_{10}+a_{n-k-1,k-1}c_{11}+
 	\cdots +a_{n-k-1,0}c_{1,k})+\cdots 
 	\end{eqnarray*}
\vspace{-20pt}
\begin{eqnarray*}	
\lefteqn{ \cdots +(a_{0k}c_{n-k,0}+a_{0,k-1}c_{n-k,1}+\cdots +a_{00}c_{n-k,k})}\\
 	&\phantom{====}& = \sum\limits_{\substack{p=0}}^{k}\sum\limits_{\substack{j=k}}^{n}a_{n-j,p}c_{j-k,k-p}
\end{eqnarray*}
completing the proof of the lemma.
\end{proof}

The relationship between $A_n$ and $C_n$ is given by the following lemma.

\begin{lem}\label{Contractivity and Positivity}
	If $A_n$ and $C_n$ are  defined as above, then   
	\[\left(
	{
	\begin{array}{ccccccc}
		I & C_{1}^* & C_{2}^* & \cdots & C_{n}^* \\
		C_1 & I & C_{1}^* & \cdots & C_{n-1}^*\\ 
		C_2 & C_1 & I & \cdots & C_{n-2}^*\\
		\vdots & \vdots & \vdots & \ddots & \vdots \\
		C_n & C_{n-1} & C_{n-2} & \cdots & I \\
	\end{array} }
	\right)\]
	 is non-negative definite if and only if 
	 \[\left\| \left(
	{
	\begin{array}{ccccc}
		A_1 & A_2 & A_3 &\cdots & A_n \\
		0 & A_1 & A_2 & \cdots & A_{n-1}\\ 
		0 & 0 & A_1 & \cdots & A_{n-2}\\
		\vdots & \vdots & \vdots & \ddots & \vdots \\
		0 & 0 & 0 & \cdots & A_1 \\
	\end{array} }
	\right)\right\|
	\leq 1.\]
\end{lem}
\begin{proof}
For each $n\in\mathbb N,$ 
$C_n$ commutes with $C_m$ and $A_m$ 
for all $m\in\mathbb{N}$ and 
hence we can adapt the proof of Lemma \ref{contractivity vs positivity} 
to complete the proof in this case.  
\end{proof}
 An application of the spectral theorem along with  Lemma \ref{Contractivity and Positivity}  
 gives an alternative proof of the Kor\'{a}nyi--Puk\'{a}nszky theorem as shown below.
\begin{proof}[\textbf{Proof of the Kor\'{a}nyi--Puk\'{a}nszky theorem}] The operators $I\otimes B^*$ and $B^*\otimes B^*$ 
are commuting unitaries 
and they have $\T^2$ as their joint spectrum.  
Applying spectral theorem and maximum modulus principle, 
we get the following: 
\begin{equation}\label{Spectral}
\|\chi\circ f(I\otimes B^*,B^*\otimes B^*)\|=\|\chi\circ f\|_{\D^2,\infty}.
\end{equation}
Note that 
$\chi\circ f(I\otimes B^*,B^*\otimes B^*)=A_1\otimes B^* + A_2\otimes {B^*}^2+\cdots,$
where 
$A_n:=a_{n0}I+a_{n-1,1}B^*+\cdots +a_{0n}{B^*}^n$ 
as in Lemma \ref{Contractivity and Positivity}.
Since $\|\chi\circ f\|_{\D^2,\infty}\leq 1,$ 
it follows from \eqref{Spectral} that 
$\|\mathscr{T}(A_1,\ldots,A_n)\|\leq 1$ 
for all $n\in\N$.  
From Lemma \ref{Contractivity and Positivity}, 
we conclude that 
\begin{equation}\label{Principal block}
\left(
	{
	\begin{array}{ccccccc}
		I & C_{1}^* & C_{2}^* & \cdots & C_{n}^* \\
		C_1 & I & C_{1}^* & \cdots & C_{n-1}^*\\ 
		C_2 & C_1 & I & \cdots & C_{n-2}^*\\
		\vdots & \vdots & \vdots & \ddots & \vdots \\
		C_n & C_{n-1} & C_{n-2} & \cdots & I \\
	\end{array} }
	\right)
\end{equation} 
is non-negative definite for all $n\in\N,$ where $C_n:=c_{n0}I+c_{n-1,1}B^*+\cdots +c_{0n}{B^*}^n.$ 
Hence from Theorem \ref{Matrix phi}, 
we see that the Kor\'{a}nyi--Puk\'{a}nszky 
function $\phi$ corresponding to the array 
$\big(\!\!\big(c_{jk}\big)\!\!\big)$ is positive.

Conversely, suppose the Kor\'{a}nyi--Puk\'{a}nszky 
function $\phi$ corresponding to the array 
$\big(\!\!\big(c_{jk}\big)\!\!\big)$ is positive, 
where $c_{00}$ is assumed to be $1/2.$ 
Then from Theorem \ref{Matrix phi}, it follows that operator in \eqref{Principal block} 
is non-negative definite for all $n\in \N$. 
From Lemma \ref{Contractivity and Positivity} 
and equation \eqref{Spectral}, 
we conclude that $\|\chi^{-1}\circ g\|_{\D^2,\infty}\leq 1,$ 
where $g(z_1,z_2)=2\sum_{m,n=0}^{\infty}c_{mn}z_1^mz_2^n.$ 
This is so if and only if 
$g$ maps $\D^2$ into the right half plane $H_+.$
Hence the theorem is proved.
\end{proof}
\subsection{The case of \texorpdfstring{$n$}{TEXT} - variables:} \label{subsec5.3}
Suppose $f$ is holomorphic mapping from polydisc $\D^n$ into $H_+.$ 
Without loss of generality, we assume that $f(\boldsymbol 0)=1/2.$
Let 
\[\chi \circ f(\boldsymbol z)=\sum\limits_{k=0}^{\infty}\sum\limits_{\substack{{i_1,\ldots, i_k=1}}}^{\infty}a_{e_{i_1}+\cdots + e_{i_k}}z_{i_1}\cdots z_{i_k}^{i_k}\]
be the power series expansion of $\chi \circ f.$ 
As before, $\chi \circ f$ maps $\mathbb{D}^n$ into $\mathbb{D}$ and $a_{\boldsymbol 0}=0$. Then with the understanding that $c_{\boldsymbol 0}=1/2,$ we have 
\begin{equation*}
 	 f(\boldsymbol z) = \frac{1+\chi \circ f(\boldsymbol z)}{1-\chi \circ f(\boldsymbol z)}
 	=2\left(c_{\boldsymbol 0}+\sum\limits_{k=1}^{\infty}\sum\limits_{\substack{{i_1,\ldots, i_k=1}}}^{\infty}c_{e_{i_1}+\cdots + e_{i_k}}z_{i_1}\cdots z_{i_k}^{i_k}\right).
\end{equation*}
For $k\in \mathbb N,$ let 
\[C_k=\sum\limits_{\substack{{i_1,\ldots,i_k=1}}}^{n}c_{e_{i_1}+\cdots +e_{i_k}}
\prod_{p=1}^{k} (I^{\otimes (n-i_p)} \otimes B^{*^{\otimes (i_p-1)}}) \mbox{ and }\] 
\[A_k=\sum\limits_{\substack{{i_1,\ldots,i_k=1}}}^{n}a_{e_{i_1}+\cdots +e_{i_k}}
\prod_{p=1}^{k} (I^{\otimes (n-i_p)} \otimes B^{*^{\otimes (i_p-1)}}).\]
Computations similar to the case of $n=2,$ using  $A_k$ and $C_k,$  $k\in \mathbb N,$ we can prove a Lemma analogous to the Lemma \ref{Contractivity and Positivity}. 
Therefore, as before, using spectral theorem for the operators $I^{\otimes (n-j)} \otimes B^{*^{\otimes j}},$  $j=1,\ldots, n,$ 
we  deduce the the Kor\'{a}nyi--Puk\'{a}nszky theorem for the polydisc $\mathbb{D}^n.$ 
  
In the PhD thesis \cite{GR} of the first named author, the proof of Theorem \ref{soln in 2 var} was given using the Kor\'{a}nyi--Puk\'{a}nszky theorem. The proof of Theorem \ref{soln in 2 var} in this note does not make use of the Kor\'{a}nyi--Puk\'{a}nszky theorem. It then appears that the ideas from this proof lead to a different proof of the Kor\'{a}nyi--Puk\'{a}nszky theorem.

\subsection*{Acknowledgement:} We thank Dr. Ramiz Reza for several useful suggestions during the preparation of this article. 
We also thank the referee for bringing the paper \cite{BLTT} to our attention.

\bibliographystyle{amsalpha}\bibliography{CFIntThm}

\end{document}